\documentclass{amsart}
\usepackage{amsmath}
\usepackage{amsfonts}
\usepackage{graphicx}
\usepackage{enumitem}

\usepackage{subfigure}
\usepackage{multirow}
\usepackage{array}
\usepackage{color} 

\newtheorem{theorem}{Theorem}

\newtheorem{definition}[theorem]{Definition}

\newtheorem{lemma}[theorem]{Lemma}

\newtheorem{proposition}[theorem]{Proposition}
\newtheorem{remark}[theorem]{Remark}

\def\func#1{\mathop{\rm #1}}%

\title[Periodic hydroelastic waves II]{Periodic traveling interfacial hydroelastic waves with or without mass II:  Multiple bifurcations and ripples}

\author{Benjamin F. Akers}
\address{Department of Mathematics and Statistics,
Air Force Institute of Technology, 2950 Hobson Way, WPAFB, OH 45433 USA}
\email{benjamin.akers@afit.edu}
\author{David M. Ambrose}
\address{Department of Mathematics, Drexel University, 3141 Chestnut St., Philadelphia, PA 19104}
\email{dma68@drexel.edu}
\author{David W. Sulon}
\address{Department of Mathematics, Drexel University, 3141 Chestnut St., Philadelphia, PA 19104}
\email{dws57@drexel.edu}

\begin{document}

\begin{abstract} In a prior work, the authors proved a global bifurcation theorem for spatially periodic
interfacial hydroelastic traveling waves on infinite depth, and computed such traveling waves.  The formulation
of the traveling wave problem used both analytically and numerically allows for waves with multi-valued height.
The global bifurcation theorem required a one-dimensional kernel in the linearization of the relevant mapping, but
for some parameter values, the kernel is instead two-dimensional.  In the present work, we study these cases 
with two-dimensional kernels, which occur in resonant and non-resonant variants.  We apply an implicit function theorem argument to prove existence of traveling waves in both of these situations.  
We compute the waves
numerically as well, in both the resonant and non-resonant cases.

\end{abstract}

\maketitle

\section{Introduction}

This is a continuation of the work \cite{AkersAmbroseSulonPeriodicHydroelasticWaves}, in which the authors
studied spatially periodic traveling waves for an interfacial configuation of two irrotational, incompressible, infinitely deep
fluids, separated by a sharp interface, allowing for hydroelastic effects at the interface, with zero or positive mass density
of the elastic sheet.  The main results of \cite{AkersAmbroseSulonPeriodicHydroelasticWaves} are a global
bifurcation theorem proving existence of such traveling waves and enumerating various ways in which the branches
may end, and numerical computations of branches of these waves.  The global bifurcation theorem was based on 
an abstract result of ``identity-plus-compact'' type \cite{KielhoferBifurcationBook}, and required a one-dimensional 
kernel in the relevant linearized mapping.  However, for certain parameter values, the kernel is instead two-dimensional.
We explore these two-dimensional cases here.

The global bifurcation theorem of \cite{AkersAmbroseSulonPeriodicHydroelasticWaves} is a generalization of the 
global bifurcation theorem proved by the second author, Strauss, and Wright for vortex sheets with surface tension
\cite{AmbroseStraussWrightGlobBifurc}.  In the case of vortex sheets with surface tension, two-dimensional kernels
were also encountered, but were not investigated further.  Numerical computations of branches are not contained in 
\cite{AmbroseStraussWrightGlobBifurc}, but can be found instead in \cite{akersAmbrosePondWright}.
All of the works discussed thus far use the formulation for traveling waves introduced in 
\cite{AkersAmbroseWrightTravelingWavesSurfaceTension}, 
which allows for waves with multi-valued height by developing a traveling wave ansatz
for a parameterized curve.

The hydroelastic wave problem models the motion of a free surface which bears elastic effects; two primary
examples are ice sheets on the ocean \cite{Squire1995}, or flapping flags \cite{Alben2008}.  We use here the
Cosserat theory of elastic shells to model the elastic effects, as developed and described by Plotnikov and Toland
\cite{plotnikovToland}.  Other possible models for elastic effects are linear models or Kirchoff--Love models, but the Cosserat
theory is more suitable for large deformations, like those for which we demonstrate numerical results below.
The second author and Siegel, and Liu and the second author, have shown that the initial value problem for the
hydroelastic problem is well-posed in the case of zero mass density of the sheet and in the case of positive
mass density of the sheet \cite{ambroseSiegelWPHydro}, \cite{LiuAmbroseWellPosednessMass}.
A number of other authors have studied traveling hydroelastic waves, either rigorously or computationally, in various
cases (such as periodic or solitary waves, with or without mass, in two spatial dimensions or three spatial dimensions).
For instance, Toland and Baldi and Toland prove existence in the hydroelastic water wave case (a single fluid bounded above
by an elastic sheet) without mass \cite{toland2}, \cite{baldiToland2} and with mass along the sheet \cite{toland1}, 
\cite{baldiToland1}, and  Groves, Hewer, and Wahlen showed used variational methods to prove existence of solitary 
hydroelastic waves.  Guyenne and Parau have computed hydroelastic traveling waves including on finite depth
\cite{guyenneParauJFM}, \cite{guyenneParauFS}, and the group of Milewski, Wang, Vanden-Broeck have made a number
of computational studies for hydroelastic solitary waves on two-dimensional and three-dimensional fluids
  \cite{MVBWJFM}, \cite{MVBWPRSA469},
\cite{milewskiWangSAM},    
  Finally, we mention the work of Wang, Parau, Milewski, and Vanden-Broeck, in which solitary waves
 were computed in an interfacial hydroelastic situation  \cite{WVBMIMA}.
Of these other traveling wave results, the most relevant to the present work is \cite{baldiToland1}; there, for the
case of hydroelastic water waves (i.e., the case of a single fluid bounded above by an elastic surface, with vacuum
above the surface), Baldi and Toland use an implicit function theorem argument in the case of a non-resonant 
two-dimensional kernel in the
linearized operator.  Our main theorem follows their argument, and includes an extension to the resonant case inspired by the methods of \cite{EhrnstromEscherWahlen} (in this work, Ehrnstr\"om, Escher, and Wahl\'en prove existence results for traveling water waves with multiple critical layers).

In the resonant case of a two-dimensional kernel in the linearized operator, one may wish to consider one-dimensional families (throughout which surface tension is fixed) of traveling wave solutions that feature two resonant, leading-order modes at small amplitudes.  The resulting waves are known as 
Wilton ripples, after the work \cite{wilton}.  In addition to Wilton's original work, many asymptotic and numerical studies exist for the water wave problem and in approximate models \cite{akers2012wilton},  \cite{akers2016high}, \cite{haupt1988modeling},
\cite{mcgoldrick1970wilton}, \cite{okamoto2001mathematical}, \cite{reederShinbrot1}, 
 \cite{trichtchenko2016instability}, 
\cite{OlgaPreprint}, \cite{vanden2002wilton}. 
The authors are aware of a few works in the literature in which rigorous existence theory for Wilton ripples is 
developed for water waves.  The first of these is by Reeder and Shinbrot, for irrotational capillary-gravity water waves \cite{reederShinbrot2}.  In \cite{TolandJonesSymmetry}, Jones and Toland (see also \cite{TolandJones1}) use techniques (including those of Shearer \cite{ShearerSecondaryBifurc}) involving bifurcation at a two-dimensional kernel under symmetry.  More recently, Martin and Matioc proved existence of Wilton ripples for water waves in the case of constant vorticity \cite{martin2013existence}.  

The plan of the paper is as follows.  In Section 2, we give the governing equations for the hydroelastic wave 
initial value problem, and we develop the traveling wave ansatz.  In Section 3 we state and prove our main theorem, 
which is an existence theorem for traveling waves in both the non-resonant and resonant cases of two-dimensional kernels.
In Section 4, in the case of resonant two-dimensional kernels, we develop asymptotic expansions for the Wilton 
ripples.  In Section 5, the numerical method is described and numerical results are presented, for both the
resonant and non-resonant cases.

The authors are grateful for support from the following funding agencies.
This work was supported in part from a grant from the Office of Naval Research (ONR grant APSHEL to Dr. Akers),
and in part by a grant from the National Science Foundation 
(grant DMS-1515849 to Dr. Ambrose).  The authors are also extremely grateful to the anonymous referees, whose
detailed readings of the manuscript have surely improved its quality.

\section{Governing equations}

In this section we describe the equations for the two-dimensional interfacial hydroelastic wave system.  We first
give evolution equations, and then specialize to the traveling wave problem.

\subsection{Equations of motion}

Our problem involves the same setup as in \cite{AkersAmbroseSulonPeriodicHydroelasticWaves}, and closely follows that of \cite{AmbroseStraussWrightGlobBifurc} and \cite{LiuAmbroseWellPosednessMass}. We consider two two-dimensional fluids which are infinite in the vertical direction and periodic in the horizontal direction. \ The two fluids are separated by a one-dimensional free interface $I$; the lower fluid has mass density $\rho_1 \geq 0$, the upper fluid has mass density $\rho_2 \geq 0$ (with $\rho_{1}$ and $\rho_{2}$ not both zero), and the interface itself has mass density $\rho \geq 0$.

We assume each fluid is irrotational and incompressible; in the interior of each fluid region, the fluid's velocity $u$ is determined by the Euler equations
\begin{eqnarray*}
u_t + u \cdot \nabla u &=& - \nabla p, \\
\func{div} \left( u \right) &=& 0, \\
u &=& \nabla \phi.
\end{eqnarray*}
However, nonzero, measure-valued vorticity may be present along the interface, since $u$ is allowed to be discontinuous across $I$. We write the vorticity as $\gamma \in \mathbb{R}$ multiplied by the Dirac mass of $I$; the amplitude $\gamma$ (which may vary along $I$) is called the ``unnormalized vortex sheet-strength" \cite{AkersAmbroseSulonPeriodicHydroelasticWaves},   \cite{AmbroseStraussWrightGlobBifurc}.

With the canonical identification of our overall region $\mathbb{R}^{2}$ with $\mathbb{C}$, we parametrize $I$ as a curve 
\begin{equation*}
z\left(\alpha,t\right) = x\left(\alpha,t\right) + i y \left(\alpha,t\right),
\end{equation*}
where $\alpha$ is the spatial parameter along $I$, and $t$ represents time. We impose periodicity conditions
\begin{eqnarray}
x\left( \alpha +2\pi ,t\right) &=&x\left( a,t\right) +M
\label{periodicity_conditions_x}, \\
y\left( \alpha +2\pi ,t\right) &=&y\left( \alpha ,t\right),
\label{periodicity_conditions_y}
\end{eqnarray}
where $M>0$. The unit tangent and upward normal vectors $T$ and $N$ are (in complex form)
\begin{eqnarray*}
T &=&\frac{z_{\alpha }}{s_{\alpha }},  \\ 
N &=&i\frac{z_{\alpha }}{s_{\alpha }},
\end{eqnarray*}
where the arclength element $s_{\alpha}$ (which is the derivative of the arclength as measured from a specified point)
is defined by
\begin{equation}
s_{\alpha }^{2} =\left\vert z_{\alpha }\right\vert ^{2}=x_{\alpha
}^{2}+y_{\alpha }^{2}.  \label{s_alpha_definition}
\end{equation}
Then, we can decompose the velocity $z_t$ as
\begin{equation}
z_t = UN + VT,
\label{eq:zt}
\end{equation}
where $U, V$ respectively denote the normal and tangential velocities. Note that throughout the text, subscripts of $t$ or
$\alpha$ denote differentiation. 

We parametrize by normalized arclength; i.e. our parametrization ensures
\begin{equation}
s_{\alpha} = \sigma \left(t\right) := \frac{L\left(t \right)}{2 \pi}
\label{eq:sigmaDef}
\end{equation}
holds for all $t$, where
\begin{equation*}
L\left(t\right)=\int\nolimits_{0}^{2\pi }s_{\alpha } \text{ }d\alpha 
\end{equation*}
is the length of one period of the interface (this means $s_{\alpha}$ is constant with respect to $\alpha$). Furthermore, define the tangent angle
\begin{equation*}
\theta := \arctan\left(\frac{y_{\alpha}}{x_{\alpha}}\right);
\end{equation*}
it is clear that we can construct the curve $z$ from $\theta$ and $\sigma$ (up to one point), and that the curvature of the interface $\kappa$ is
\begin{equation*}
\kappa = \frac{\theta_{\alpha}}{s_{\alpha}}.
\end{equation*}

The normal velocity $U$ is entirely determined by the physics and geometry of the problem; specifically,
\begin{equation}
U=\func{Re}\left( W^{\ast }N\right) ,  \label{eq:normal}
\end{equation}%
where%
\begin{equation}
W^{\ast }\left( \alpha ,t\right) =\frac{1}{2\pi i}\func{PV}\int\nolimits_{%
\mathbb{R}
}\frac{\gamma \left( \alpha ^{\prime },t\right) }{z\left( \alpha ,t\right)
-z\left( \alpha ^{\prime },t\right) }d\alpha ^{\prime }
\label{def:Wstar}
\end{equation}
is the Birkhoff-Rott integral \cite{AmbroseStraussWrightGlobBifurc} (we use ${}^\ast$ to denote the complex conjugate).

However, we are able to freely choose the tangential velocity $V$ in order to enforce (\ref{eq:sigmaDef}) at all times $t$; explicitly, let $V$ be periodic and such that
\begin{equation}
V_{\alpha }=\theta _{\alpha }U-\frac{1}{2\pi }\int\nolimits_{0}^{2\pi
}\theta _{\alpha }U\text{ }d\alpha
\label{def:Valpha}
\end{equation}
holds.  By differentiating (\ref{s_alpha_definition}) with respect to $t$, it can be easily verified (as in \cite{AkersAmbroseSulonPeriodicHydroelasticWaves}) that such a choice $V$ yields (\ref{eq:sigmaDef}) for all $t$ (as long as  (\ref{eq:sigmaDef}) holds at $t=0$).

The vortex sheet-strength $\gamma \left( \alpha ,t\right) $ (which can be written in terms of the jump in tangential velocity across $I$) plays a crucial role in the interface's evolution. Using the same model as in \cite{AkersAmbroseSulonPeriodicHydroelasticWaves} (which itself is a combination of those used in \cite{ambroseSiegelWPHydro} and \cite{LiuAmbroseWellPosednessMass}), we assume the jump in pressure across $I$ to be
\begin{equation}\label{pressureJumpEquation}
\left[ \left[ p\right] \right] =\rho \left(\func{Re}\left(W_t^\ast N\right) + V_W \theta_t \right) +  \frac{1}{2} E_{b}\left( \kappa _{ss}+\frac{\kappa ^{3}}{2%
}-\tau _{1}\kappa \right) + g \rho \func{Im}N,
\end{equation}%
where
\begin{equation*} 
V_W := V - \func{Re} \left( W^{\ast} T \right),
\end{equation*} and the constants $E_{b} \geq 0$, $\tau _{1} > 0$, and $g$ are (respectively) the bending modulus, a surface tension parameter, and acceleration due to gravity \cite{AkersAmbroseSulonPeriodicHydroelasticWaves}. 
The formula \eqref{pressureJumpEquation} is derived in detail in the appendix (Section 8) of \cite{ambroseSiegelWPHydro}
in the case without mass along the sheet.
This is found to be consistent with the model of Plotnikov and Toland \cite{plotnikovToland}
using the special Cosserat theory of elastic shells, and accounting for mass yields \eqref{pressureJumpEquation};
note that Plotnikov and Toland treat this case of positive mass.

From this, we may write an equation determining the evolution of $\gamma$ \cite{AkersAmbroseSulonPeriodicHydroelasticWaves}, \cite{LiuAmbroseWellPosednessMass}:%
\begin{eqnarray}
\;\;\;\; \gamma _{t} &=&-\frac{\widetilde{S}}{\sigma ^{3}}\left( \partial _{\alpha
}^{4}\theta +\frac{3\theta _{\alpha }^{2}\theta _{\alpha \alpha }}{2}-\tau
_{1}\sigma ^{2}\theta _{\alpha \alpha }\right) +\frac{\left( V_{W}\gamma
\right) _{\alpha }}{\sigma }-2\widetilde{A}\left( \func{Re}\left( W_{\alpha
t}^{\ast }N\right) \right) \label{penultimateGammatEqn}\\
&&-\left( 2A-\frac{2\widetilde{A}\theta _{\alpha }}{%
\sigma }\right) \left( \func{Re}\left( W_{t}^{\ast }T\right) \right)
\sigma -2\widetilde{A}\left( \left( V_{W}\right) _{\alpha }\theta
_{t}+V_{W}\theta _{t\alpha }+\frac{gx_{\alpha \alpha }}{\sigma }\right) \notag \\
&&-2A\left( \frac{\gamma \gamma _{\alpha }}{4\sigma ^{2}}-V_{W}\func{Re}\left(
W_{\alpha }^{\ast }T\right) +gy_{\alpha }\right), \notag
\end{eqnarray}%
where
\begin{equation*}
\widetilde{S} :=\frac{E_{b}}{\rho _{1}+\rho _{2}} \geq 0,
\end{equation*}
\begin{equation*}
A :=\frac{\rho _{1}-\rho _{2}}{\rho _{1}+\rho _{2}}\in [-1,1] ,
\end{equation*}
is the Atwood number, and
\begin{equation*}
\widetilde{A} := \frac{\rho }{\rho _{1}+\rho _{2}} \geq 0.
\end{equation*}%
Setting $S := \widetilde{S} \mathbin{/} \left| g \right|$, we can non-dimensionalize (\ref{penultimateGammatEqn}), and write (as in \cite{AkersAmbroseSulonPeriodicHydroelasticWaves})
\begin{eqnarray}
\;\;\;\;\;\; \gamma _{t} &=&-\frac{S}{\sigma ^{3}}\left( \partial _{\alpha }^{4}\theta +%
\frac{3\theta _{\alpha }^{2}\theta _{\alpha \alpha }}{2}-\tau _{1}\sigma
^{2}\theta _{\alpha \alpha }\right) +\frac{\left( V_{W}\gamma \right)
_{\alpha }}{\sigma }  \label{eq:gammat} \\
&&-2\widetilde{A}\left( \func{Re}\left( W_{\alpha t}^{\ast }N\right) \right)
-\left( 2A-\frac{2\widetilde{A}\theta _{\alpha }}{\sigma }\right) \left( 
\func{Re}\left( W_{t}^{\ast }T\right) \right) \sigma  \notag \\
&&-2\widetilde{A}\left( \left( V_{W}\right) _{\alpha }\theta
_{t}+V_{W}\theta _{t\alpha }+\frac{x_{\alpha \alpha }}{\sigma }\right)
-2A\left( \frac{\gamma \gamma _{\alpha }}{4\sigma ^{2}}-V_{W}\func{Re}\left(
W_{\alpha }^{\ast }T\right) +y_{\alpha }\right).  \notag
\end{eqnarray}%

Thus, in the two-dimensional hydroelastic vortex sheet problem with mass, equations (\ref{eq:zt}), (\ref{eq:normal}), (\ref{def:Wstar}), (\ref{def:Valpha}), and (\ref{eq:gammat}) together determine the motion of the interface \cite{AkersAmbroseSulonPeriodicHydroelasticWaves}.

\subsection{Traveling wave ansatz}

We are specifically interested in traveling wave solutions to the two-dimensional hydroelastic vortex sheet problem with mass. 

\begin{definition}
Suppose $\left(z,\gamma\right)$ is a solution to (\ref{eq:zt}), (\ref{eq:normal}), (\ref{def:Wstar}), (\ref{def:Valpha}), and (\ref{eq:gammat}), and additionally satisfies
\begin{equation}
\label{eq:travelingWaveAssumption}
\left(z,\gamma\right)_t = \left(c,0\right)
\end{equation}
for some parameter $c \in \mathbb{R}$.  Then, we say $\left(z,\gamma\right)$ is a traveling wave solution to  (\ref{eq:zt}), (\ref{eq:normal}), (\ref{def:Wstar}), (\ref{def:Valpha}), and (\ref{eq:gammat}) with speed $c$.
\end{definition}
\begin{remark}
The values $c$ and  $\tau_1$ will serve as the bifurcation parameters for our analysis in Section \ref{sec:bifurcThm}.
\end{remark}
Assuming (\ref{eq:travelingWaveAssumption}), the following clearly hold:
\begin{eqnarray}
\label{Uequation}U &=& -c \sin \theta, \\
\nonumber
V &=& c \cos \theta, \\
\nonumber
V_W &=& c \cos \theta - \func{Re}\left( W^{\ast }T\right), \\
\nonumber
x_{\alpha} &=& \sigma \cos \theta, \\
\nonumber
y_{\alpha} &=& \sigma \sin \theta, \\
\nonumber
\theta_t &=& 0.
\end{eqnarray}
(Note that while these equations do follow from the traveling wave ansatz \eqref{eq:travelingWaveAssumption}, a reader 
interested in further
details of the calculatuions could consult \cite{AkersAmbroseWrightTravelingWavesSurfaceTension}, \cite{AmbroseStraussWrightGlobBifurc}.)  Perhaps less clear is that \eqref{Uequation} implies, upon differentiation
with respect to $t,$ that a traveling wave must satisfy
\begin{equation}\label{gammaTEqualsZero}
\gamma_{t}=0.
\end{equation}
(Again, further details may be found in \cite{AkersAmbroseWrightTravelingWavesSurfaceTension}, and especially
 \cite{AmbroseStraussWrightGlobBifurc}.)
Also, both $W_t^\ast$ and $W_{\alpha t}^\ast$ vanish under this assumption (this can be shown by carefully differentiating under the principal value integral). Thus, noting these facts, alongside
\begin{equation*}
\left( c\cos \theta -\func{Re}\left( W^{\ast }T\right) \right) \left( \func{%
Re}\left( W_{\alpha }^{\ast }T\right) \right) =-\frac{1}{2}\partial _{\alpha
}\left\{ \left( c\cos \theta -\func{Re}\left( W^{\ast }T\right) \right)
^{2}\right\} ,
\end{equation*} (which can be computed from the above, as in \cite{AmbroseStraussWrightGlobBifurc}) we can combine (\ref{eq:gammat}) and \eqref{gammaTEqualsZero} as follows:
\begin{eqnarray}
0 &=&-\frac{S}{\sigma ^{3}}\left( \partial _{\alpha }^{4}\theta +\frac{%
3\theta _{\alpha }^{2}\theta _{\alpha \alpha }}{2}-\tau _{1}\sigma
^{2}\theta _{\alpha \alpha }\right)  \label{eq:gammaEqBeforeSigmaTauDivision} \\
&&-2\widetilde{A}\left( \cos \theta \right) _{\alpha }  
+\frac{\left( \left( c\cos \theta -\func{Re}\left( W^{\ast }T\right)
\right) \gamma \right) _{\alpha }}{\sigma }  \notag \\
&&-A\left( \frac{\partial _{\alpha }\left( \gamma ^{2}\right) }{4\sigma ^{2}}%
+2\sigma \sin \theta +\partial _{\alpha }\left\{ \left( c\cos \theta -\func{%
Re}\left( W^{\ast }T\right) \right) ^{2}\right\} \right);  \notag
\end{eqnarray}%
more details on these basic calculations can be found in \cite{AkersAmbroseSulonPeriodicHydroelasticWaves}. Following the conventions of \cite{AmbroseStraussWrightGlobBifurc} and \cite{AkersAmbroseSulonPeriodicHydroelasticWaves}, we label
\begin{multline}
\Omega \left( \theta ,\gamma ;c,\sigma \right) :=\left(
\left( c\cos \theta -\func{Re}\left( W^{\ast }T\right) \right) \gamma
\right) _{\alpha }\label{def:Omega}  \\
-A\left[ \frac{\left( \gamma ^{2}\right) _{\alpha }}{%
4\sigma }+2\sigma ^{2}\sin \theta +\sigma \partial _{\alpha }\left\{ \left(
c\cos \theta -\func{Re}\left( W^{\ast }T\right) \right) ^{2}\right\} \right],
\end{multline}
\begin{equation}
\Xi \left( \theta ;\tau_1,\sigma \right):= \frac{3}{2}\theta _{\alpha
}^{2}\theta _{\alpha \alpha } -\tau _{1}\sigma ^{2}\theta
_{\alpha \alpha } + \frac{2\widetilde{A}\sigma ^{3}}{S%
}\left( \cos \theta \right) _{\alpha } \label{def:Xi}
\end{equation}
(note that, as a matter of notation, $\Omega$ differs from its analogue in \cite{AmbroseStraussWrightGlobBifurc} and \cite{AkersAmbroseSulonPeriodicHydroelasticWaves} by a factor of $\tau_1$).  Then, by multiplying both sides of (\ref{eq:gammaEqBeforeSigmaTauDivision}) by $\sigma^3 \mathbin{/} S$, we can write (\ref{eq:gammaEqBeforeSigmaTauDivision}) as
\begin{equation}
0=\partial _{\alpha }^{4}\theta +\Xi \left( \theta ;\tau_1, \sigma \right) -\frac{\sigma ^{2}}{S}\Omega \left( \theta ,\gamma ;c,\sigma \right) .
\label{eq:gammaEq1}
\end{equation}

Recall that that the evolution of the interface is also determined by (\ref{eq:normal}). Under the traveling wave assumption, this equation becomes
\begin{equation}
0=c\sin \theta +\func{Re}\left( W^{\ast }N\right) .  \label{eq:thetaEq1}
\end{equation}

Finally, note that we can construct $z$ from $\theta$ (uniquely, up to rigid translation) via the equation
\begin{equation}
z\left( \alpha ,0\right) =z\left( \alpha ,t\right) -ct=\sigma
\int\nolimits_{0}^{\alpha }\exp \left( i\theta \left( \alpha ^{\prime
}\right) \right) \text{ }d\alpha ^{\prime };  \label{z_construction}
\end{equation}%
thus, by using (\ref{z_construction}) to produce $z$ as it appears in the Birkhoff-Rott integral (\ref{def:Wstar}), we may consider equations (\ref{eq:gammaEq1}), (\ref{eq:thetaEq1}) as being purely in terms of $\left(\theta,\gamma\right)$. However, another modification of (\ref{eq:gammaEq1}), (\ref{eq:thetaEq1})  is required to ensure that $2 \pi$-periodic solutions $\left(\theta,\gamma\right)$ yield (via (\ref{z_construction})) \emph{periodic} traveling wave solutions $\left(z,\gamma\right)$ of (\ref{eq:zt}), (\ref{eq:normal}), and (\ref{eq:gammat}). 

\subsection{Periodicity considerations}\label{subsec:perConsid}
We continue to apply the same reformulation procedure as in \cite{AkersAmbroseSulonPeriodicHydroelasticWaves}. Throughout Section \ref{subsec:perConsid}, assume that $\theta$ is $2\pi$-periodic and four times differentiable. Define
\begin{equation*}
\overline{\cos \theta }:=\frac{1}{2\pi }\int\nolimits_{0}^{2\pi }\cos
\left( \theta \left( \alpha ^{\prime }\right) \right) d\alpha ^{\prime },%
\text{ \ \ \ }\overline{\sin \theta }:=\frac{1}{2\pi }\int\nolimits_{0}^{2%
\pi }\sin \left( \theta \left( \alpha ^{\prime }\right) \right) d\alpha
^{\prime };
\end{equation*}%
throughout, we use this overline notation to indicate similar average quantities. Then, given $M>0$ and $\theta $ with nonzero $\overline{\cos \theta }$, put
\begin{equation*}
\widetilde{Z}\left[ \theta \right] \left( \alpha \right) :=\frac{M}{2\pi 
\overline{\cos \theta }}\left[ \int\nolimits_{0}^{\alpha }\exp \left(
i\theta \left( \alpha ^{\prime }\right) \right) \text{ }d\alpha ^{\prime
}-i\alpha \overline{\sin \theta }\right] ;
\end{equation*}%
we call $\widetilde{Z}\left[ \theta \right]$ the ``renormalized curve."  By direct calculation, we may show that $\widetilde{Z}\left[ \theta \right] $ satisfies the $M$-periodicity condition (\ref{periodicity_conditions_x}), (\ref{periodicity_conditions_y}), i.e.
\begin{equation*}
\widetilde{Z}\left[ \theta \right] \left( \alpha +2\pi \right) =\widetilde{Z}%
\left[ \theta \right] \left( \alpha \right) +M;
\end{equation*}%
also, $\widetilde{Z}\left[ \theta \right]$ is one derivative smoother than $\theta$. The normal and tangent vectors to $\widetilde{Z}\left[ \theta \right]$ are clearly
\begin{eqnarray*}
\widetilde{T}\left[ \theta \right] &=&\frac{\partial _{\alpha }\widetilde{Z}%
\left[ \theta \right] }{\left\vert \partial _{\alpha }\widetilde{Z}\left[
\theta \right] \right\vert }, \\
\widetilde{N}\left[ \theta \right] &=&i\frac{\partial _{\alpha }\widetilde{Z}%
\left[ \theta \right] }{\left\vert \partial _{\alpha }\widetilde{Z}\left[
\theta \right] \right\vert }.
\end{eqnarray*}%
We use the following form of the Birkhoff-Rott integral, defined for real-valued $%
\gamma $ and complex-valued $\omega$ that satisfies $\omega \left( \alpha
+2\pi \right) =\omega \left( \alpha \right) +M$:
\begin{equation}
\label{eq:birkhoffRottDef}
B\left[ \omega \right] \gamma \left( \alpha \right) :=\frac{1}{2iM}\func{PV}%
\int\nolimits_{0}^{2 \pi}\gamma \left( \alpha ^{\prime }\right) \cot \left( \frac{\pi }{M}\left(
\omega \left( \alpha \right) -\omega \left( \alpha ^{\prime }\right) \right)
\right) d\alpha ^{\prime }.
\end{equation}%
Then, $W^{\ast }=B\left[ z\right] \gamma $ by Mittag-Leffler's well-known cotangent series expansion (see, e.g., \cite{AblowitzFokasComplexVar}, Chapter 3). Finally, define the ``renormalized" version of (\ref{def:Omega}):
\begin{eqnarray*}
\widetilde{\Omega }\left( \theta ,\gamma ;c\right)
&:=&\partial _{\alpha }\left\{ \left( c\cos \theta -\func{Re}%
\left( \left(B\left[ \widetilde{Z}\left[ \theta \right] \right] \gamma\right)  
\widetilde{T}\left[ \theta \right] \right) \right) \gamma \right\} \\
&&-A\left( \frac{\pi \overline{\cos \theta }}{2M}\partial
_{\alpha }\left( \gamma ^{2}\right) +\frac{M^{2}}{2\pi ^{2}\left( \overline{%
\cos \theta }\right) ^{2}}\left( \sin \theta -\overline{\sin \theta }\right)
\right) \\
&&-A\left( \frac{M}{2\pi \overline{\cos \theta }}\partial
_{\alpha }\left\{ \left( c\cos \theta -\func{Re}\left( \left(B\left[ \widetilde{Z}%
\left[ \theta \right] \right] \gamma\right) \widetilde{T}%
\left[ \theta \right] \right) \right) ^{2}\right\} \right) .
\end{eqnarray*}

With these definitions at hand, we are able to re-write (\ref{eq:gammaEq1}), (\ref{eq:thetaEq1}), ensuring $M$-periodicity in a
traveling-wave  solution to these equations. The following proposition appears almost verbatim in \cite{AkersAmbroseSulonPeriodicHydroelasticWaves}, and an analogous version is proved in \cite{AmbroseStraussWrightGlobBifurc}:
\begin{proposition}
\label{spacial_periodic_version}Suppose $c\neq 0$ and $2\pi $-periodic
functions $\theta ,\gamma $ satisfy $\overline{\cos \theta }\neq 0$ and%
\begin{equation}
\partial _{\alpha }^{4}\theta +\Xi \left( \theta ;\tau_1,\sigma \right) -\frac{%
\sigma ^{2}}{S}\widetilde{\Omega }\left( \theta ,\gamma ;c\right) =0,
\label{eq:thetaEq2}
\end{equation}%
\begin{equation}
\func{Re}\left( \left(B\left[ \widetilde{Z}\left[ \theta \right] \right] \gamma\right)
 \widetilde{N}\left[ \theta \right] \right) +c\sin
\theta =0,  \label{eq:gammaEq2}
\end{equation}
with $\sigma =M \mathbin{/} \left( 2\pi \overline{\cos \theta }\right) $. Then, $%
\left( \widetilde{Z}\left[ \theta \right] \left( \alpha \right) +ct,\gamma
\left( \alpha \right) \right) $ is a traveling wave solution to (\ref{eq:gammaEq1}) and (\ref{eq:thetaEq1}) with speed $c$, and $%
\widetilde{Z}\left[ \theta \right] \left( \alpha \right) +ct$ is spatially
periodic with period $M$. 
\end{proposition}
\begin{remark}
Under the assumptions of Proposition \ref{spacial_periodic_version}, we have $\widetilde{\Omega }\left( \theta ,\gamma ;c\right) = \Omega \left( \theta ,\gamma ; c,\sigma \right)$ with $\sigma = M\mathbin{/}\left( 2\pi \overline{\cos \theta }\right) $. This result is the major component of the above proposition's proof \cite{AmbroseStraussWrightGlobBifurc}.
\end{remark}
Our equations require one last reformulation before an application of bifurcation theory techniques.

\subsection{Final reformulation}
\label{subsec:finalReform}

We would like to write (\ref{eq:thetaEq2}), (\ref{eq:gammaEq2}) in an ``identity plus compact" (over an appropriately chosen domain) form. In order to ``solve" (\ref{eq:thetaEq2}) for $\theta $, we introduce two operators.  Throughout, let $\nu$ be a general $2 \pi$-periodic map with convergent Fourier series, and write $\nu \left( \alpha \right) =\sum\nolimits_{k=-\infty }^{\infty }%
\widehat{\nu }\left( k\right) \exp \left( ik\alpha \right) $. First, define the projection
\begin{equation}
\label{eq:projDef}
P\nu (\alpha) :=\nu \left(\alpha \right)-\frac{1}{2\pi }%
\int\nolimits_{0}^{2\pi } \nu \left( \alpha^\prime \right) d\alpha^\prime;
\end{equation}
by definition, $P\nu$ has mean zero. Then, define an inverse derivative operator $\partial _{\alpha }^{-4}$ in the Fourier space:
\begin{equation}
\label{eq:inverseDerivDef}
\begin{cases}
\widehat{\partial _{\alpha }^{-4}\nu }\left( k\right) :=k^{-4} \widehat{\nu }\left( k\right), & k \neq 0 \\
             \widehat{\partial _{\alpha }^{-4}\nu }\left( 0\right) := 0
\end{cases}.
\end{equation}
Clearly, this operator $\partial _{\alpha }^{-4}$ preserves periodicity, and for sufficiently smooth, periodic $\nu$ with mean zero, 
\begin{equation*}
\partial _{\alpha }^{-4}\partial _{\alpha }^{4}\nu= \nu = \partial _{\alpha }^{4}\partial _{\alpha }^{-4}\nu.
\end{equation*}

Now, as in \cite{AkersAmbroseSulonPeriodicHydroelasticWaves}, we apply $\partial _{\alpha }^{-4} P$ to both sides of (\ref{eq:thetaEq2}), yielding
\begin{equation}
\label{eq:thetaEq3}
0=P\theta +\partial _{\alpha }^{-4}  \Xi \left( \theta ;\tau_1, \sigma \right)
-\frac{\sigma ^{2}}{S}\partial _{\alpha }^{-4}  \widetilde{\Omega }%
\left( \theta ,\gamma ;c\right) 
\end{equation}%
(note that $P\Xi =\Xi $ and $P \widetilde{\Omega } = \widetilde{\Omega }$).  For convenience, define
\begin{eqnarray*}
\gamma _{1}&:=& P \gamma, \\
\overline{\gamma } &:=& \gamma -  P \gamma;
\end{eqnarray*}%
we may then decompose $\gamma = \gamma_1 + \overline{\gamma }$, and the mean $\overline{\gamma }$ may be specified \emph{a priori} as another constant quantity in our problem \cite{AmbroseStraussWrightGlobBifurc}. We may then put
\begin{equation}
\label{def:Theta}
\Theta \left( \theta ,\gamma _{1};c,\tau_1\right) :=\frac{\sigma ^{2}}{S}%
\partial _{\alpha }^{-4} \widetilde{\Omega}\left( \theta ,\gamma _{1}+%
\overline{\gamma };c\right) -\partial _{\alpha }^{-4}  \Xi \left( \theta
;\tau_1,\sigma \right) ,
\end{equation}%
and concisely write (\ref{eq:thetaEq3}) as
\begin{equation}
\label{eq:thetaEq4}
P\theta -\Theta \left( \theta ,\gamma _{1};c, \tau_1\right) =0.
\end{equation}%

The Birkhoff-Rott integral (\ref{eq:birkhoffRottDef}) appears in both equations (\ref{eq:thetaEq4}) and (\ref{eq:gammaEq2}). Before our reformulation of (\ref{eq:gammaEq2}), we note that $B\left[\omega\right]\gamma$
can be written as the sum
\begin{equation}
\label{eq:BirkRottDecomp}
B\left[ \omega \right] \gamma = \frac{1}{2 i \omega_{\alpha}} H \gamma + \frac{1}{2 i}\left[H,\frac{1}{\omega_{\alpha}}\right]\gamma + \mathcal{K}\left[\omega\right] \gamma,
\end{equation}
where
\begin{equation} \label{eq:hilbertTransDef}
H \gamma (\alpha) := \frac{1}{2 \pi}\, \func{PV}\int \nolimits_0^{2 \pi} \gamma \left( \alpha^\prime \right) \cot \left(\frac{1}{2}\left( \alpha - \alpha^\prime \right) \right) d \alpha^\prime
\end{equation}
is the Hilbert transform for periodic functions, the operator
\begin{equation*}
\left[H,\psi\right]\gamma := H\left(\psi \gamma\right) - \psi H \gamma
\end{equation*}
is the commutator of the Hilbert transform and multiplication by a smooth function, and the remainder is
\begin{eqnarray*}
&&\mathcal{K} \left[ \omega \right] \gamma \left( \alpha \right) :=  \\
&&\frac{1}{2 i M}\func{PV}\int\nolimits_{0}^{2\pi }\gamma \left( \alpha
^{\prime }\right) \left[ \cot \left( \frac{\pi}{M}\left( \omega \left( \alpha
\right) -\omega \left( a^{\prime }\right) \right) \right) -\frac{M/\left(2 \pi\right) }{\omega_{\alpha ^{\prime }} \left( \alpha ^{\prime }\right) }\cot \left( \frac{1}{2}%
\left( \alpha -\alpha ^{\prime }\right) \right) \right] d\alpha ^{\prime }.
\end{eqnarray*}%
The most singular part of (\ref{eq:BirkRottDecomp}) is the term $\left(2 i \omega_{\alpha}\right)^{-1} H\gamma$, while the other two terms, which we denote
\begin{equation*}
\mathcal{K}^\prime \left[\omega\right] \gamma := \frac{1}{2 i}\left[H,\frac{1}{\omega_{\alpha}}\right]\gamma + \mathcal{K}\left[\omega\right]\gamma,
\end{equation*}
are sufficiently smooth on the domain we define in Section \ref{subsec:mappingProperties} 
\cite{AmbroseWellPosedness2003}, \cite{AmbroseZeroSurfaceTensionLimit2DDarcy}.

Then, noting that $H^2 \gamma = -P\gamma = -\gamma_1$, we may (as in \cite{AmbroseStraussWrightGlobBifurc}) write (\ref{eq:gammaEq2}) as
\begin{equation}
\label{eq:gammaEq3}
\gamma _{1}-H\left\{ 2\left\vert \partial _{\alpha }\widetilde{Z}\left[
\theta \right] \right\vert \func{Re}\left(\left(\mathcal{K}^\prime \left[ \widetilde{Z}\left[
\theta \right] \right] \left( \overline{\gamma }+\gamma _{1}\right) \right) 
\widetilde{N}\left[ \theta \right] \right) +2c\left\vert \partial _{\alpha }%
\widetilde{Z}\left[ \theta \right] \right\vert \sin \theta \right\}=0,
\end{equation}
and substitute $\theta = \Theta \left( \theta ,\gamma _{1};c, \tau_1\right)$ from (\ref{eq:thetaEq4}), finally yielding
\begin{equation}
\gamma _{1}-\Gamma \left( \theta ,\gamma _{1};c, \tau_1\right) =0,
\label{eq:gammaEq4}
\end{equation}%
where%
\begin{eqnarray}
\Gamma \left( \cdot \right) &:=& H\left\{2\left\vert \partial
_{\alpha }\widetilde{Z}\left[ \Theta \left( \cdot\right) %
\right] \right\vert \func{Re}\left( \left( \mathcal{K}^\prime \left[ \widetilde{Z}\left[
\Theta \left( \cdot\right) \right] \right] \left( \overline{%
\gamma }+\gamma _{1}\right) \right) \widetilde{N}\left[ \Theta \left( \cdot\right) \right] \right)\right.\label{def:Gamma} \\ 
&&+\left.2c\left\vert \partial _{\alpha }\widetilde{Z}\left[ \Theta \left( \cdot\right) \right] \right\vert \sin \left( \Theta \left( \cdot\right) \right) \right\}, \notag
\end{eqnarray}
and $\left(\cdot \right)$ indicates suppression of the arguments $\left( \theta ,\gamma _{1};c, \tau_1\right) $.

For concision, define the two-dimensional parameter $\mu :=\left(
c,\tau _{1}\right) $. Denote 
\begin{equation*}
F\left( \theta ,\gamma _{1};\mu \right)
:=\left( \theta -\Theta \left( \theta ,\gamma _{1};\mu \right) ,\gamma
_{1}-\Gamma \left( \theta ,\gamma _{1};\mu \right) \right),
\end{equation*}
 where $\Theta
,\Gamma $ are as defined earlier. In Section \ref{subsec:mappingProperties}, we define spaces over which the equation $F\left( \theta ,\gamma _{1};\mu \right) = 0$ is equivalent to (\ref{eq:thetaEq4}), (\ref{eq:gammaEq4}) and discuss the mapping properties of $F$.

\subsection{Mapping properties}\label{subsec:mappingProperties}
We introduce the following spaces, using the same notation and language as in \cite{AkersAmbroseSulonPeriodicHydroelasticWaves}:
\begin{definition}
Let $H_{\text{per}}^{s}$ denote $H_{\text{per}}^{s}\left[ 0,2\pi \right] ,$
i.e. the usual Sobelev space of $2\pi $-periodic functions from $%
\mathbb{R}
$ to $%
\mathbb{C}
$ with square-integrable derivatives up to order $s\in 
\mathbb{N}
$. Let $H_{\text{per,odd}}^{s}$ denote the subset of $H_{\text{per}}^{s}$
comprised of odd functions; define $H_{\text{per,even}}^{s}$ similarly. \
Let $H_{\text{per,}0\text{,even}}^{s}$ denote the subset of $H_{\text{%
per,even}}^{s}$ comprised of mean-zero functions. Finally, letting $H_{%
\text{loc}}^{s}$ denote the usual Sobolev space of functions in $H^{s}\left(
I\right) $ for all bounded intervals $I$, we put%
\begin{equation*}
H_{M}^{s}=\left\{ \omega \in H_{\text{loc}}^{s}:\omega \left( \alpha \right)
-\frac{M\alpha }{2\pi }\in H_{\text{per}}^{s}\right\} .
\end{equation*}
For $b\geq 0$ and $s\geq 2$, define the ``chord-arc space"%
\begin{equation*}
C_{b}^{s}=\left\{ \omega \in H_{M}^{s}:\inf_{\alpha ,\alpha ^{\prime }\in 
\left[ 0,2 \pi \right] }\left\vert \frac{\omega \left( \alpha \right) -\omega
\left( \alpha ^{\prime }\right) }{\alpha ^{\prime }-\alpha }\right\vert
>b\right\} .
\end{equation*}%
\end{definition}
\begin{remark}\label{rem:RemainderAndCommutatorReg}
As proved in Lemma 3.5 of \cite{AmbroseWellPosedness2003}, the above chord-arc condition ensures that the ``remainder" mapping $\left(w,\gamma\right) \mapsto \mathcal{K} \left[ \omega \right] \gamma$ is a smooth map from $C_b^s \times H_{\text{per}}^1 \rightarrow H_{\text{per}}^{s-1}$, with
\begin{equation*}
\left\lVert\mathcal{K} \left[ \omega \right] \gamma \right\rVert_{H_{\text{per}}^{s-1}} \leq c_1 \left\lVert \gamma \right\rVert_{H_{\text{per}}^1} \exp \left\{c_2 \left\lVert \omega \right\rVert_{H^{s} \left(0, 2\pi\right)}\right\}.
\end{equation*}Furthermore, in the case $\omega \in H^3$ (i.e. $1/\omega_{\alpha} \in H^2$), Lemma 4 of \cite{AmbroseZeroSurfaceTensionLimit2DDarcy} provides the following estimate for the commutator term of (\ref{eq:BirkRottDecomp}):
\begin{equation*}
\left\lVert \left[H,\frac{1}{\omega_{\alpha}}\right] \gamma \right\rVert_{H_{\text{per}}^2} \leq c_1 \left\lVert \gamma \right\rVert_{H_{\text{per}}^1} \left\lVert \frac{1}{\omega_{\alpha}} \right\rVert_{H^2 \left(0, 2\pi\right)} .
\end{equation*}Explicitly, in the notation of this lemma of \cite{AmbroseZeroSurfaceTensionLimit2DDarcy}, we use $j=1$, $n=2$.  

These regularity results are stated for the case $M=2\pi$; however, this is not an essential assumption by any means.
\end{remark}
We can now define a domain over which $\left(\Theta,\Gamma \right)$ is compact. We state the following result, which is proved in \cite{AkersAmbroseSulonPeriodicHydroelasticWaves} and relies in part on the regularity results discussed in Remark \ref{rem:RemainderAndCommutatorReg} (note in particular that these results are applied for $\omega = \widetilde{Z}\left[ \theta \right]$, which is one derivative smoother than $\theta \in H^2$).
\begin{proposition}
\label{mapping_proposition}Define
\begin{equation*}
X:=H_{\text{per},\text{odd}}^{2}\times H_{\text{per},0,\text{even}}^{1} \times \mathbb{R} \times \left(0,\infty\right)
\end{equation*}
and 
\begin{equation*}
U_{b,h}:=\left\{ \left( \theta ,\gamma _{1};c, \tau_1\right) \in X
:\overline{\cos \theta }>h,\widetilde{Z}\left[ \theta \right] \in C_{b}^{2}%
\text{ and } \widetilde{Z}\left[ \Theta \left( \theta ,\gamma _{1};c,\tau_1\right) %
\right] \in C_{b}^{5}\right\},
\end{equation*}%
with the usual topology inherited from $X$.  The mapping $\left( \Theta ,\Gamma \right) $ from $U_{b,h}\subseteq X$ into 
$X$ is compact.
\end{proposition}

Thus, the mapping $F$, defined over $X$, is of the form ``identity plus compact," and the equation
\begin{equation}
\label{eq:FequalsZero}
F\left( \theta ,\gamma _{1};\mu \right) = 0
\end{equation}
is equivalent to (\ref{eq:thetaEq4}) (note that since $\theta$ is odd, $P\theta = \theta$) and (\ref{eq:gammaEq4}). In the next section, we calculate the linearization $L(\mu)$ of $F$; the status of $L(\mu)$ as a Fredholm operator will allow us to apply the Lyapunov--Schmidt procedure to our main equation (\ref{eq:FequalsZero}) and subsequently prove a bifurcation theorem with methods akin to \cite{baldiToland1}, \cite{EhrnstromEscherWahlen}.

\section{Bifurcation theorem}\label{sec:bifurcThm}

\subsection{Linearization and its kernel}
\label{subsec:linearization}

Our bifurcation analysis begins with a calculation of the linearization $L(\mu)$ of $F$ at $\left( 0,0;\mu \right) $.  This is shown with
greater detail in \cite{AkersAmbroseSulonPeriodicHydroelasticWaves}; we simply state the results here. Letting $%
\left( \overrightarrow{\theta },\overrightarrow{\gamma }\right) $ denote the
direction of differentiation (i.e. for any sufficiently regular map $\nu $, $%
D\nu :=\left. D_{\theta ,\gamma _{1}}\nu \left( \theta ,\gamma _{1};\mu\right)
\right\vert _{\left( 0,0;\mu\right) }:=\lim_{\varepsilon \rightarrow 0}\frac{1%
}{\varepsilon }\left( \nu \left( \varepsilon \overrightarrow{\theta }%
,\varepsilon \overrightarrow{\gamma };\mu\right) -\nu \left(0,0;\mu\right)\right) $), we
define:%
\begin{equation*}
L(\mu)%
\begin{bmatrix}
\overrightarrow{\theta } \\ 
\overrightarrow{\gamma }%
\end{bmatrix}%
:=%
\begin{bmatrix}
\overrightarrow{\theta }-D\Theta \\ 
\overrightarrow{\gamma }-D\Gamma%
\end{bmatrix}%
\end{equation*}%
where%
\begin{equation*}
L(\mu)
\begin{bmatrix}
\overrightarrow{\theta } \\ 
\overrightarrow{\gamma }%
\end{bmatrix}%
:=
\begin{bmatrix}
L_{11} & L_{12} \\ 
L_{21} & L_{22}%
\end{bmatrix}%
\begin{bmatrix}
\overrightarrow{\theta } \\ 
\overrightarrow{\gamma }%
\end{bmatrix}%
\end{equation*}%
with%
\begin{eqnarray*}
L_{11} &:=&1-\frac{\overline{\gamma }M}{4\pi S}\left( 
\overline{\gamma }-\frac{cAM}{\pi }\right) \partial
_{\alpha }^{-4}\partial _{\alpha }H+\frac{AM^{4}}{8\pi ^{4}S}\partial
_{\alpha }^{-4}P-\frac{\tau _{1}M^{2}}{4\pi ^{2}}\partial _{\alpha
}^{-4}\partial _{\alpha }^{2}, \\
L_{12} &:=&\frac{M^{2}}{4\pi ^{2}S}\left( \frac{\pi A\overline{%
\gamma }}{M}-c\right) \partial _{\alpha
}^{-4}\partial _{\alpha }, \\
L_{21} &:=&-\frac{c\overline{\gamma }M^{2}}{4\pi ^{2}S}\left( \overline{\gamma} -\frac{cAM}{\pi }\right) H\partial
_{\alpha }^{-4}\partial _{\alpha }H+\frac{cAM^{5}}{8\pi ^{5}S}H\partial
_{\alpha }^{-4}P-\frac{c\tau _{1}M^{3}}{4\pi ^{3}}H\partial _{\alpha
}^{-4}\partial _{\alpha }^{2}, \\
L_{22} &:=&1+\frac{c M^{2}}{4\pi ^{2}S}\left( A\overline{%
\gamma }-\frac{cM}{\pi }\right) H\partial _{\alpha
}^{-4}\partial _{\alpha }.
\end{eqnarray*}

In \cite{AkersAmbroseSulonPeriodicHydroelasticWaves}, we then proceed to calculate the Fourier coefficients of $%
L(\mu)$, which we also simply list below. Using the same notation as in Section \ref{subsec:finalReform}, we calculate
\begin{eqnarray*}
\widehat{L_{11}}\left( k\right) &=&1-\frac{\overline{\gamma }M}{%
4\pi S}\left( \overline{\gamma }-\frac{cAM}{\pi }%
\right) \frac{1}{\left\vert k\right\vert ^{3}}+\frac{AM^{4}}{8\pi ^{4}S}%
\frac{1}{k^{4}}+\frac{\tau _{1}M^{2}}{4\pi ^{2}}\frac{1}{k^{2}}, \\
\widehat{L_{12}}\left( k\right) &=&i\frac{M^{2}}{4\pi ^{2}S}\left( 
\frac{\pi A\overline{\gamma }}{M}-c\right) \frac{1%
}{k^{3}}, \\
\widehat{L_{21}}\left( k\right) &=&i\frac{c\overline{\gamma }M^{2}}{%
4\pi ^{2}S}\left( \overline{\gamma }-\frac{cAM}{\pi }\right) \frac{1}{k^{3}}-i\frac{cAM^{5}}{8\pi ^{5}S}\frac{\text{sgn}%
\left( k\right) }{k^{4}}-i\frac{c\tau _{1}M^{3}}{4\pi ^{3}}\frac{\text{sgn}%
\left( k\right) }{k^{2}}, \\
\widehat{L_{22}}\left( k\right) &=&1+\frac{cM^{2}}{4\pi ^{2}S}%
\left( A\overline{\gamma }-\frac{cM}{\pi}\right) 
\frac{1}{\left\vert k\right\vert ^{3}};
\end{eqnarray*}%
and note%
\begin{equation*}
\widehat{L(\mu)%
\begin{bmatrix}
\overrightarrow{\theta } \\ 
\overrightarrow{\gamma }%
\end{bmatrix}%
}\left( k\right)=%
\begin{bmatrix}
\widehat{L_{11}}\left( k\right) & \widehat{L_{12}}\left( k\right) \\ 
\widehat{L_{21}}\left( k\right) & \widehat{L_{22}}\left( k\right)%
\end{bmatrix}%
\begin{bmatrix}
\overrightarrow{\theta } \\ 
\overrightarrow{\gamma }%
\end{bmatrix}%
.
\end{equation*}

We summarize the spectral information of $L(\mu)$ in the following
proposition, which appears nearly verbatim (albeit with proof) in \cite{AkersAmbroseSulonPeriodicHydroelasticWaves}. 

\begin{proposition}
\label{spectral_proposition}Let $L(\mu)$ be the linearization of $F$ at $\left( 0,0;\mu \right)$. The spectrum of $L(\mu)$
is the set of eigenvalues $\left\{ 1\right\} \cup \left\{ \lambda _{k}\left(
\mu \right) :k\in 
\mathbb{N}
\right\} $, where%
\begin{equation}
\lambda _{k}\left( \mu \right) :=1+\frac{M^{2}\tau _{1}}{4\pi ^{2}}k^{-2}+\frac{%
-c^{2}M^{3}+2Ac\overline{\gamma }M^{2}\pi -\overline{\gamma }^{2}M\pi ^{2}}{%
4\pi ^{3}S}k^{-3}+\frac{AM^{4}}{8\pi ^{4}S}k^{-4}.  \label{lambda_c_def}
\end{equation}%
Each eigenvalue $\lambda $ of $L(\mu)$ has algebraic multiplicity equal to
its geometric multiplicity, which we denote 
\begin{equation*}
N_{\lambda }\left( \mu \right) :=\left\vert \left\{ k\in 
\mathbb{N}
:\lambda _{k}\left(\mu \right) =\lambda \right\} \right\vert ,
\end{equation*}%
and the corresponding eigenspace is%
\begin{equation*}
E_{\lambda }\left( \mu \right) :=\textup{span}\left\{ 
\begin{bmatrix}
-\frac{\pi }{cM}\sin \left( k\alpha \right) \\ 
\cos \left( k\alpha \right)%
\end{bmatrix}%
:k\in 
\mathbb{N}
\text{ such that }\lambda _{k}\left( \mu \right) =\lambda \right\} .
\end{equation*}%

Further, define the polynomial
\begin{equation}
R\left( k ; \tau_1\right) :=AM^{4}+2\left(\left(A^{2}-1\right)%
\overline{\gamma }^{2}M\pi ^{3}\right) k+2M^{2}\pi ^{2}S\tau _{1}k^{2}+8\pi
^{4}Sk^{4}.
\label{eq:RkDef}
\end{equation}%
For fixed $k$, if the inequality 
\begin{equation}
R\left( k ; \tau_1\right) \geq 0
\label{nonstrict_ineq_condition}
\end{equation}%
holds, then the values of $c\in 
\mathbb{R}
$ for which $\lambda _{k}\left(c,\tau_1 \right) =0$ are 
\begin{equation}
c_{\pm }\left( k ; \tau_1 \right) :=\frac{A\overline{\gamma}\pi }{M}\pm \sqrt{\frac{R\left( k ; \tau_1\right)}{2kM^{3}\pi }},  \label{c_pm}
\end{equation}%
and this zero eigenvalue has multiplicity $N_{0}\left( c_{\pm }\left(
k; \tau_1 \right),\tau_1 \right) \leq 2$. Specifically, if we define the polynomial
\begin{equation}
p\left( l,k;\tau_1\right) :=-AM^{4}+2kl\pi ^{2}S\left( 4\left(
k^{2}+kl+l^{2}\right) \pi ^{2}+M^{2}\tau _{1}\right) ,
\label{eq:plkDef}
\end{equation}%
$p\left( \cdot,k;\tau_1 \right) $ has a single real root (denoted $l\left( k\right) 
$), and we have $N_{0}\left( c_{\pm }\left(
k; \tau_1 \right),\tau_1 \right) =2$ if and only
if $l\left( k\right) $ is a positive integer not equal to $k$.
\end{proposition}

\begin{remark}
In \cite{AkersAmbroseSulonPeriodicHydroelasticWaves}, we consider wave speeds $c_{\pm }\left( k; \tau_1\right) $ for which 
\begin{equation*}
N_{0}\left( c_{\pm }\left(k; \tau_1 \right),\tau_1 \right) =1.
\end{equation*}
This condition is necessary for the linearization $L(\mu)$ to have an \textquotedblleft odd
crossing number\textquotedblright\ at $c=c_{\pm }\left( k; \tau_1\right)$ (see
\cite{AkersAmbroseSulonPeriodicHydroelasticWaves}, Proposition 13); subsequently, the linearization's possession of an
odd crossing number is a hypothesis for the abstract bifurcation theorem we
apply (see \cite{AkersAmbroseSulonPeriodicHydroelasticWaves}, Theorem 8, which itself is a slight modification of
that which appears in \cite{KielhoferBifurcationBook}). However, neither \cite{AkersAmbroseSulonPeriodicHydroelasticWaves} nor
the relevant theorem in \cite{KielhoferBifurcationBook} draw conclusions of bifurcation in
the $N_{0}\left( c_{\pm }\left(k; \tau_1 \right),\tau_1 \right) =2$ case, since the
crossing number cannot be odd in this case. Yet, bifurcation may still
occur when $N_{0}\left( c_{\pm }\left(
k; \tau_1 \right),\tau_1 \right) =2$; it is this
case that comprises the main focus of our analysis here.
\end{remark}

We are now ready to state our main theorem, which (like the analogous results in \cite{baldiToland1}, \cite{EhrnstromEscherWahlen}) provides conditions in
which bifurcation can be concluded in the $N_{0}\left( c_{\pm }\left(
k; \tau_1 \right),\tau_1 \right) =2$ case. 

\subsection{Main theorem}

Recall the definition of the space $X$ from Proposition \ref{mapping_proposition}.

\begin{theorem}\label{thm:main}Define the polynomials $p$ and $R$ as in (\ref{eq:plkDef}) and (\ref{eq:RkDef}).  Let $b>0$ and $h>0$ be arbitrary, and  suppose that $k < l$ are a pair of positive integers that satisfy the following criteria for some $\tau_1^*>0$: \\
\begin{enumerate}[label={\upshape(\roman*)}]
\item $p(l,k;\tau_1^*)=0$ (2D null space condition), and \\
\item $R(k;\tau_1^*) > 0$ (non-degeneracy condition).\\
\end{enumerate}
Given $\mu^* := \left(c_{\pm }\left(
k ; \tau_1^* \right),\tau_1^*  \right)$, let
\begin{eqnarray*}
\mathcal{V}:= \func{Ker} L(\mu^*),&\;&  \mathcal{R} := \func{Range} L(\mu^*), \\
\mathcal{V^\dagger}:= \func{Ker} L^{\dagger} (\mu^*),&\;&  \mathcal{R^\dagger}:= \func{Range} L^{\dagger}(\mu^*),
\end{eqnarray*}where $L^{\dagger}(\mu^*)$ is the Hermitian adjoint of $L(\mu^*)$.  Further, let $\Pi_\mathcal{V^\dagger}$, $\Pi_\mathcal{R}$ denote the corresponding projections.  Finally, denote basis elements of $\mathcal{V}$ by
\begin{equation*}
v_j \left(\alpha\right) := \begin{bmatrix}
-\frac{\pi }{c_{\pm }\left(
k ; \tau_1^* \right) M}\sin \left( j\alpha \right) \\ 
\cos \left( j\alpha \right)%
\end{bmatrix}, \;\;\; j \in \left\{k,l\right\}.
\end{equation*}

Depending on whether $l \mathbin{/} k \in \mathbb{N}$, one of the following alternatives hold.

\begin{description}
\item[Non-resonant case] Suppose $l \mathbin{/} k \notin \mathbb{N}$.  Then, there exist neighborhoods $\mathcal{N}_{t},\,\mathcal{N}_{\mu}\subseteq\mathbb{R}^2$ of (respectively) $(0,0)$ and $\mu^*$, neighborhoods $\mathcal{N_V},\,\mathcal{N_{R^{\dagger}}}$ of $0$ in $\mathcal{V}$ and $\mathcal{R^\dagger}\cap U_{b,h}$ (respectively) along with smooth functions $\overline{\mu}:\mathcal{N}_{t} \rightarrow \mathcal{N}_{\mu}$, $\overline{y}:\mathcal{N_V} \times \mathcal{N}_{\mu} \rightarrow \mathcal{N_{R^\dagger}}$, and
\begin{equation*}
w(t_1,t_2):= t_1 v_k + t_2 v_l +\overline{y} \left( t_1 v_k + t_2 v_l, \overline{\mu}(t_1,t_2) \right),
\end{equation*}with
\begin{equation*}
\overline{\mu}(0,0) = \mu^*,
\end{equation*}
\begin{equation*}
\overline{y} \left( t_1 v_k + t_2 v_l, \overline{\mu}(t_1,t_2) \right) = O(t_1^2 + t_2^2),
\end{equation*}
such that
\begin{equation*}
F\left(w(t_1,t_2);\overline{\mu}(t_1,t_2)\right)=0
\end{equation*}
for all $(t_1,t_2) \in \mathcal{N}_{t}$.  

The maps $\overline{\mu}, \, \overline{y}$ are unique in the sense that if $\Pi_\mathcal{V^{\dagger}} F\left(w\left(t_1,t_2\right); \mu \right) = 0$ for $\left(t_1,t_2\right) \in \mathcal{N}_t$ with $t_1 \neq 0$ and $t_2 \neq 0$, then $\mu = \overline{\mu}\left(t_1,t_2\right)$, and likewise if $\Pi_\mathcal{R}F\left( v+y,\mu \right) =0$ with $\left( v,\mu \right) \in \mathcal{N_V} \times \mathcal{N}_{\mu}$ and $y\in \mathcal{N_{R^\dagger}}$, then $y=\overline{y}\left( v;\mu \right)$.

\item[Resonant case]  Suppose $l \mathbin{/} k \in \mathbb{N}$.  Given $\delta > 0$, there exist neighborhoods $\mathcal{N}_{r} \subseteq \mathbb{R}^{+},\, \mathcal{N}_{\mu}\subseteq\mathbb{R}^2$ of (respectively) $0$ and $\mu^*$, neighborhoods $\mathcal{N_V},\,\mathcal{N_{R^{\dagger}}}$ of $0$ in $\mathcal{V}$ and $\mathcal{R^\dagger}\cap U_{b,h}$ (respectively) along with smooth functions $\overline{\mu }:\mathcal{N}_{r}\times \left( \left(\delta ,\pi -\delta \right) \cup \left( -\pi +\delta, -\delta \right)\right) \rightarrow N_{\mu }$, $\overline{y}:\mathcal{N_V} \times \mathcal{N}_{\mu} \rightarrow \mathcal{N_{R^\dagger}}$, and
\begin{eqnarray*}
w(r,\beta)&:=& \left(r \cos \beta\right) v_k + \left(r \sin \beta\right) v_l \\
&&+\overline{y} \left( \left(r \cos \beta\right) v_k + \left(r \sin \beta\right) v_l, \overline{\mu}(r,\beta) \right),
\end{eqnarray*}with
\begin{equation*}
\overline{\mu}(0,\beta) = \mu^*,
\end{equation*}
\begin{equation*}
\overline{y} \left( \left(r \cos \beta\right) v_k + \left(r \sin \beta\right) v_l, \overline{\mu}(r,\beta) \right) = O(r^2),
\end{equation*}
such that
\begin{equation*}
F\left(w(r,\beta);\overline{\mu}(r,\beta)\right)=0
\end{equation*}
for all $r \in \mathcal{N}_{r}$ and for all $\beta$ satisfying $\delta <\left\vert
\beta \right\vert <\pi -\delta $.

The map $\overline{\mu}$ is unique in the sense that if $\Pi_\mathcal{V^{\dagger}} F\left(w\left(r \cos \beta,r \sin \beta \right); \mu \right) = 0$ for nontrivial $r \in \mathcal{N}_r$ and $\beta \in \left(\delta ,\pi -\delta \right) \cup \left(-\pi +\delta, -\delta \right) \setminus \left\{\pi/2,-\pi/2\right\}$, then $\mu = \overline{\mu}\left(r,\beta \right)$.  The map $\overline{y}$ is unique in the same sense as in the non-resonant case.

\end{description}

\end{theorem}

\begin{remark}
The non-degeneracy condition (ii) is also required in \cite{AkersAmbroseSulonPeriodicHydroelasticWaves} to prove bifurcation when $N_{0}\left( c_{\pm }\left(
k; \tau_1 \right),\tau_1 \right) =1$.  Thus, we extend some of the results of \cite{AkersAmbroseSulonPeriodicHydroelasticWaves} to the $N_{0}\left( c_{\pm }\left(
k; \tau_1 \right),\tau_1 \right) =2$ case (although different methods are used here than in \cite{AkersAmbroseSulonPeriodicHydroelasticWaves}).  Indeed,
if $R\left(k;\tau_1\right)>0$, the linear problem has traveling waves, and if one uses linear stability analysis on the flat state, the spectrum is pure imaginary.  The quantity $R\left(k;\tau_1\right)$ is negative only when the mean shear $\overline{\gamma}$ is too large for the parameter choice, which results in a linear problem which does not support traveling waves.
\end{remark}

The results of Theorem \ref{thm:main} can be summarized in the following informal, non-technical fashion:

\begin{theorem}[Main theorem, non-technical version]
Define the mapping $F$ and parameter $\mu$ as above.  If conditions (i), (ii) of Theorem \ref{thm:main} are met for some pair of integers $k,l$ at $\tau_1 = \tau_1^\ast,$ then there exists a smooth sheet of solutions $\left(\theta, \gamma_1; \mu \right)$ to the traveling wave problem
\begin{equation*}
F\left(\theta, \gamma_1; \mu \right) = 0
\end{equation*}
bifurcating from the trivial solution $\left(0,0;\mu\right)$ at $\mu=\mu^\ast$ (where $\mu^\ast$ is as defined above).
\end{theorem}

We prove Theorem \ref{thm:main} throughout the remainder of this section.  The proof begins with an application of the classical Lyapunov--Schmidt reduction to our equations; this reduction process utilizes the implicit function theorem (an interested reader might consult, for example, Chapter 13 of \cite{hunterNachtergaele} for a statement of the implicit function theorem). Then, the implicit function theorem is employed again to solve the reduced equations.  We proceed to apply the Lyapunov--Schmidt process to (\ref{eq:FequalsZero}).

\subsubsection{Lyapunov--Schmidt reduction} \label{subsubsec:lyapunovSchmidt}

Throughout, fix $b,\,h > 0$, and let $\mu = \left(c, \tau_1\right)$ be our two-dimensional parameter such that at a special value $\tau_1^\ast > 0$, there are two distinct integer solutions $k$, $l$ of $p\left( l,k;\tau_1^\ast\right) =0$, where $p$ is as defined in Proposition \ref{spectral_proposition}. Also, without loss of generality, we take $k < l$.  As Proposition \ref{spectral_proposition} indicates, the linearization $L\left(\mu\right)$ has a two-dimensional kernel at $\mu=\mu^\ast$, where $\mu^* := \left(c_{\pm }\left(k ; \tau_1^* \right),\tau_1^*  \right)$.  Let $v$ be an element of this 2D null
space $\mathcal{V}:=\func{ker} L \left( \mu^* \right) $, i.e.%
\begin{equation}\label{vParameterized}
v\left( \alpha \right) =t_{1}%
\begin{bmatrix}
-\frac{\pi }{c_{\pm }\left(k ; \tau_1^* \right) M}\sin \left( k\alpha \right) \\ 
\cos \left( k\alpha \right)%
\end{bmatrix}%
+t_{2}%
\begin{bmatrix}
-\frac{\pi }{c_{\pm }\left(k ; \tau_1^* \right) M}\sin \left( l\alpha \right) \\ 
\cos \left( l\alpha \right)%
\end{bmatrix}%
.
\end{equation}%

Let $L^{\dagger} (\mu^*)$ denote the Hermitian adjoint of $L (\mu^*)$, and put
\begin{eqnarray*}
\mathcal{V^\dagger}&:=& \func{Ker} L^{\dagger} (\mu^*),\\  
\mathcal{R^\dagger}&:=& \func{Range} L^{\dagger}(\mu^*),
\end{eqnarray*}
Note that from the results in Section \ref{subsec:mappingProperties}, we have that $L\left(\mu\right)$ is Fredholm. As in \cite{baldiToland1} (and in general for Lyapunov--Schmidt reductions), write
each $w$ in the domain $U_{b,h}$ as%
\begin{equation*}
w=v+y,
\end{equation*}%
where $y\in \mathcal{R^\dagger}\cap U_{b,h}$. With the notation $\mathcal{R} := \func{Range} L(\mu^*)$, we can write the entire spaces as the direct sums%
\begin{eqnarray*}
X &=& \mathcal{V^\dagger} \oplus \mathcal{R}, \\
U_{b,h} &=&\mathcal{V}\oplus \left( \mathcal{R^\dagger}\cap U_{b,h}\right) ,
\end{eqnarray*}%
and let $\Pi_\mathcal{V^\dagger},\Pi_\mathcal{R}$ be the projections onto $\mathcal{V^\dagger}$ and $\mathcal{R}$, respectively. \
Our primary equation%
\begin{equation*}
F\left( w;\mu \right) =0
\end{equation*}%
is then equivalent to the system%
\begin{equation*}
\begin{cases}
\Pi_\mathcal{V^\dagger}F\left( v+y;\mu \right) =0, \\
\Pi_\mathcal{R}F\left( v+y;\mu \right) =0.
\end{cases}
\end{equation*}%
The former is the bifurcation equation; the latter is the ``auxiliary"
equation. We prove a lemma analogous to Lemma 6.1 of \cite{baldiToland1}, which
concerns the auxiliary equation.

\begin{lemma}
\label{baldi-toland-like-lemma}For fixed $b, h > 0$, there exist neighborhoods $\mathcal{N_V}$ of $0 \in \mathcal{V}$, $\mathcal{N}_{\mu}$ of $\mu^* \in \mathbb{R}^2$, and $\mathcal{N_{R^\dagger}}$ of $0 \in \mathcal{R^\dagger}\cap U_{b,h}$, and a smooth function $\overline{y}:\mathcal{N_V} \times \mathcal{N}_{\mu} \rightarrow \mathcal{N_{R^\dagger}}$
such that%
\begin{equation*}
\Pi_\mathcal{R}F\left( v+\overline{y}\left( v;\mu \right) ;\mu \right) =0
\end{equation*}%
for all $\left( v,\mu \right) \in \mathcal{N_V} \times \mathcal{N}_{\mu}$. In these neighborhoods $\mathcal{N_V} \times \mathcal{N}_{\mu}$ and $\mathcal{N_{R^\dagger}}$, $\overline{y}$ is unique
(i.e. if $\Pi_\mathcal{R}F\left( v+y,\mu \right) =0$ with $\left( v,\mu \right) \in \mathcal{N_V} \times \mathcal{N}_{\mu}$ and $y\in \mathcal{N_{R^\dagger}}$, then $y=%
\overline{y}\left( v;\mu \right)$). Also, for all $\left( 0,\mu \right) \in \mathcal{N_V} \times \mathcal{N}_\mu$, 
\begin{eqnarray}
\overline{y}\left( 0;\mu \right) &=0 \label{baldiTolandLikeLemmaFirstResult},\\
D\overline{y}\left( \mu \right) &=0 \label{baldiTolandLikeLemmaSecondResult},\\
\partial _{c}\overline{y}\left( 0;\mu \right) &=0 \label{baldiTolandLikeLemmaThirdResult}, \\
\partial _{\tau_1}\overline{y}\left( 0;\mu \right) &=0, \label{baldiTolandLikeLemmaFourthResult}
\end{eqnarray}
(where, as before, $D$ denotes the Fr\'{e}chet derivative about $0$). 
\end{lemma}

\begin{proof}
Define
\begin{eqnarray*}
G &:& \mathcal{V} \times \mathcal{R^\dagger} \rightarrow \mathcal{R}, \\
G\left( v,y\right) &=& \Pi_\mathcal{R}F\left( v+y;\mu ^{\ast }\right) ,
\end{eqnarray*}%
so
\begin{equation*}
G_{y}\left( 0,0\right) y=\Pi_\mathcal{R}L\left( \mu ^{\ast }\right) y,
\end{equation*}%
(throughout, let a subscript of ``$y$" or ``$v$" denote the corresponding Fr\'{e}chet derivative). 

Let $y \in \mathcal{R^\dagger}\cap U_{b,h}$ be nontrivial. We have%
\begin{equation*}
\Pi_\mathcal{R} L\left( \mu ^{\ast }\right) y \neq 0\text{,}
\end{equation*}%
since $y$ is not in the null space of $L\left( \mu ^{\ast }\right) $ and
$L\left( \mu ^{\ast }\right) y$ is in $\mathcal{R}$ by definition. Also, the operator $G_{y}\left( 0,0\right)$ is clearly surjective on $\mathcal{R}$.  Thus, we can apply the
implicit function theorem to produce such a (unique) $\overline{y}$ in the
given neighborhood. Note that%
\begin{equation*}
G\left( 0,0\right) =0,
\end{equation*}%
so, if $v=0$, we find $\overline{y}\left( 0;\mu \right) =0$ as well. 

For the next result (\ref{baldiTolandLikeLemmaSecondResult}), we have, by the implicit function theorem,
\begin{equation*}
G\left( v,\overline{y}\left( v;\mu \right) \right) =0
\end{equation*}%
in the appropriate neighborhoods. Then, by applying the chain rule (and noting the
previous result that $\overline{y}\left( 0;\mu \right) =0)$, we have for all 
$v\in \mathcal{V}$,%
\begin{equation}
G_{v}\left( 0,0\right) v+G_{y}\left( 0,0\right) \overline{y}_{v}\left( 0;\mu
\right) v=0.  \label{chain_rule_implicit}
\end{equation}%
Examine the following:
\begin{equation*}
G_{v}\left( 0,0\right) v =\Pi_\mathcal{R}L\left( \mu ^{\ast }\right) v
=0,
\end{equation*}%
since $v$ is in the null space of $L\left( \mu ^{\ast }\right) $. \
Therefore, the first term on the left-hand side of (\ref{chain_rule_implicit}) vanishes, and we
are left with the equation%
\begin{equation*}
G_{y}\left( 0,0\right) \overline{y}_{v}\left( 0;\mu \right) v=0
\end{equation*}%
(for all $v\in \mathcal{V}$). For our application of the implicit function theorem, we demonstrated the nonsingularity of $%
G_{y}\left( 0,0\right) $  (which hence admits only the trivial solution in $\mathcal{R^\dagger}$); thus,
$\overline{y}_{v}\left( 0;\mu \right) v=0$ (again, for all $v\in \mathcal{V}$). Since the domain
of $\overline{y}$ is $\mathcal{V}$, we have (for all $\mu \in \mathcal{N}_\mu$) that $\overline{y}%
_{v}\left( 0;\mu \right) $ is the zero operator. This is precisely (\ref{baldiTolandLikeLemmaSecondResult}). 

The two results (\ref{baldiTolandLikeLemmaThirdResult}), (\ref{baldiTolandLikeLemmaFourthResult}) follow by differentiating (\ref{baldiTolandLikeLemmaFirstResult}) with
respect to $c$ or $\tau_1$. 
\end{proof}

Now, define%
\begin{equation}
\label{def:Phi}
\Phi \left( t_{1},t_{2};\mu \right) :=\Pi_\mathcal{V^\dagger}F\left( v+\overline{y}\left(
v;\mu \right) ;\mu \right)
\end{equation}
(this is the Lyapunov--Schmidt reduction). 
(Notice that $t_{1}$ and $t_{2}$ do not appear in the right-hand side of \eqref{def:Phi}; however, recall that
$v$ may be specified by $t_{1}$ and $t_{2}$ as in \eqref{vParameterized}.)

We can clearly decompose $\Phi $
further. A simple calculation indicates that, for general integers $j>0$, the eigenfunction $\phi _{j}\left( \cdot ;\mu
\right) \in H_{\text{per,odd}}^{2}\times H_{\text{per,0,even}}^{1}$ of the adjoint $%
L^{\dagger }\left( \mu \right) \,$corresponding to eigenvalue $\lambda
_{j}\left( \mu \right) $ may be written as%
\begin{equation*}
\phi _{j}\left( \alpha ;\mu \right) :=%
\begin{bmatrix}
-\frac{n\left( j;\mu \right) \sin \left(j\alpha\right) }{2\pi ^{2}j} \\ 
\left( -cM+A\overline{\gamma }\pi \right) \cos \left(j\alpha\right)%
\end{bmatrix},
\end{equation*}%
where%
\begin{equation}
\label{def:nFromAdjointEigenfcns}
n\left( j;\mu \right) =AM^{3}+\left( 2Ac\overline{\gamma }M\pi ^{2}-2%
\overline{\gamma }^{2}\pi ^{3}\right) j+2M\pi ^{2}S\tau _{1} j^{2}.
\end{equation}
Clearly, the null space of $L^{\dagger }\left( \mu ^{\ast }\right) $ is
spanned by $\left\{ \phi _{k}\left( \cdot ;\mu ^{\ast }\right) ,\phi
_{l}\left( \cdot ;\mu ^{\ast }\right) \right\}$; let $\Pi _{k}$, $\Pi _{l}$ denote
the projections onto each one-dimensional subspace spanned by $\phi _{k}\left( \cdot ;\mu ^{\ast }\right)$, $\phi _{l}\left( \cdot ;\mu ^{\ast }\right)$, respectively. For general $x \in X$, these projections may be written as
\begin{eqnarray*}
\Pi _{k} x &=& \left<x, \phi _{k}\left( \cdot ;\mu ^{\ast }\right)\right> \mathbin{/} \left<\phi _{k}\left( \cdot ;\mu ^{\ast }\right), \phi _{k}\left( \cdot ;\mu ^{\ast }\right)\right>,\\
\Pi _{l} x &=& \left<x, \phi _{l}\left( \cdot ;\mu ^{\ast }\right)\right> \mathbin{/} \left<\phi _{l}\left( \cdot ;\mu ^{\ast }\right), \phi _{l}\left( \cdot ;\mu ^{\ast }\right)\right>,
\end{eqnarray*}%
where $\left<\cdot,\cdot\right>$ denotes the usual inner product on $L^2 \times L^2$. Then, we can write the bifurcation equation as the system%
\begin{equation}
\label{eq:bifurcEqs}
\begin{cases}
0 =\Phi _{k}\left( t_{1},t_{2};\mu \right) :=\Pi _{k}\Phi \left(
t_{1},t_{2};\mu \right) =\Pi _{k}F\left( v+\overline{y}\left( v;\mu \right)
;\mu \right), \\
0 =\Phi _{l}\left( t_{1},t_{2};\mu \right) :=\Pi _{l}\Phi \left(
t_{1},t_{2};\mu \right) =\Pi _{l}F\left( v+\overline{y}\left( v;\mu \right)
;\mu \right).
\end{cases}
\end{equation}

We proceed to prove another result analogous to that which appears in \cite{baldiToland1}.

\begin{lemma}
\label{smoothness_of_Phi_lemma}
The first bifurcation equation satisfies
\begin{equation}
\label{firstFromLemma}\Phi _{k}\left( 0,t_{2};\mu \right) =0 \;\;\;\text{ for all }t_{2}\text{, }\mu. \\
\end{equation}
Furthermore, if $l \mathbin{/} k \notin \mathbb{N}$, then
\begin{equation}
\label{secondFromLemma}\Phi _{l}\left( t_{1},0;\mu \right) = 0 \;\;\; \text{ for all }t_{1}\text{, }\mu
\end{equation}
additionally holds.
\end{lemma}

\begin{proof}
Define $Z_{l}$ to be the closure of 
\begin{equation*}
\text{span}\left\{ 
\begin{bmatrix}
-\frac{\pi }{c_\pm\left(k; \tau_1^\ast \right)M}\sin \left( jl\alpha \right) \\ 
\cos \left( jl\alpha \right)%
\end{bmatrix}%
:j\in 
\mathbb{N}
\right\}
\end{equation*}%
in $L^{2}\left( 0,2 \pi \right) \times L^{2}\left( 0,2 \pi \right) $; i.e. $Z_{l}$ is the
subpace of $2\pi / l$-periodic functions in $L^{2}\left( 0,2 \pi \right) \times
L^{2}\left( 0,2 \pi \right) $. Apply Lemma \ref{baldi-toland-like-lemma}, albeit replace domain $U_{b,h}$
with $Z_{l}$ and use
\begin{equation*}
v=t_{2}%
\begin{bmatrix}
-\frac{\pi }{c_\pm\left(k; \tau_1^\ast \right)M}\sin \left( l\alpha \right) \\ 
\cos \left( l\alpha \right)%
\end{bmatrix}.%
\end{equation*}%
By uniqueness of the $\overline{y}$ produced by Lemma \ref{baldi-toland-like-lemma}, we have $\overline{y}\left( v;\mu
\right) \in \mathcal{R^\dagger} \cap Z_{l}$, and hence $v+\overline{y}\left( v;\mu \right) $ is 
$2\pi /l$-periodic. Given that $F$ preserves periodicity, we have that $%
F\left( v+\overline{y}\left( v;\mu \right) \right) $ is $2\pi /l$-periodic
as well. Then, recalling $k < l$ and projecting onto $\phi _{k}\left( \cdot ;\mu ^{\ast }\right)$, we have that $\Phi _{k}\left( 0,t_{2};\mu \right) =0$.

If the non-resonance condition ($l \mathbin{/} k \notin \mathbb{N}$) holds, we can show $\Phi
_{l}\left( t_{1},0;\mu \right) =0$ by a similar argument.  This condition is needed because we would not expect the $2\pi /k$-periodic $F\left( v+\overline{y}\left( v;\mu \right) \right) $ to be orthogonal to $\phi _{l}\left( \cdot ;\mu ^{\ast }\right)$ if $l$ is a multiple of $k$.
\end{proof}

With these results at hand, we wish to eventually look for solutions to the bifurcation equations (\ref{eq:bifurcEqs}). To do so, we employ a method which closely follows that of \cite{baldiToland1} and \cite{EhrnstromEscherWahlen}: we first manipulate the bifurcation equations, then apply the implicit function theorem to an equivalent problem.

\subsubsection{Solving the reduced equations}
For now, assume $l \mathbin{/} k \notin \mathbb{N}$; we later adjust the argument to cover the resonant case.  As in \cite{baldiToland1}, define%
\begin{equation}
\label{defPsi}
\Psi :=\left( \Psi _{k},\Psi _{l}\right) ,
\end{equation}%
where 
\begin{eqnarray}
\Psi _{k}\left( t_{1},t_{2};\mu \right) &:=&\int\nolimits_{0}^{1}\partial
_{t_{1}}\Phi _{k}\left( xt_{1},t_{2};\mu \right) dx,  \label{defPsi1} \\
\Psi _{l}\left( t_{1},t_{2};\mu \right) &:=&\int\nolimits_{0}^{1}\partial
_{t_{2}}\Phi _{l}\left( t_{1},xt_{2};\mu \right) dx.  \label{defPsi2}
\end{eqnarray}
As in \cite{baldiToland1}, solving the bifurcation equations (\ref{eq:bifurcEqs}) is equivalent to solving
\begin{equation*}
\left(\Psi_k\left( t_{1},t_{2};\mu \right), \Psi_l \left( t_{1},t_{2};\mu \right)\right) = \left(0,0\right).
\end{equation*}

\begin{lemma}
\label{lem:relnPhiPsi}
Assume $l \mathbin{/} k \notin \mathbb{N}$.  For nontrivial $ t_{1}$, $\Phi_{k}\left( t_{1},t_{2};\mu \right) = 0$ if and only if $\Psi _{k}\left( t_{1},t_{2};\mu \right) = 0$ (with the same result for $\Phi _{l}$, $\Psi _{l}$ over nontrivial $t_2$).
\end{lemma}
\begin{proof}
Note that, for $t_1 \neq 0$,
\begin{equation}
\Psi _{k}\left( t_{1},t_{2};\mu \right) =\frac{1}{t_{1}}\Phi _{k}\left( t_{1},t_{2};\mu \right) \label{relationship_between_Phi_Psi1}
\end{equation}%
(this can be shown by using the substitution $u=t_1 x$); likewise, we have (for $t_2 \neq 0$)
\begin{equation}
\Psi _{l}\left( t_{1},t_{2};\mu \right) =\frac{1}{t_{2}}\Phi _{l}\left(
t_{1},t_{2};\mu \right) .  \label{relationship_between_Phi_Psi2}
\end{equation} 
\end{proof}

We next show smoothness of $\Psi _{k},\Psi _{l}$ (provided that the non-resonance condition holds).  This result essentially follows directly from Lemma \ref{smoothness_of_Phi_lemma}.

\begin{lemma}
Provided $l \mathbin{/} k \notin \mathbb{N}$, the mappings $\Psi _{k}$ and $\Psi _{l}$, as defined in (\ref{defPsi1}) and (\ref%
{defPsi2}), are smooth.
\label{smoothness_of_Psi_lemma}
\end{lemma}

\begin{proof} From \eqref{relationship_between_Phi_Psi1} and \eqref{relationship_between_Phi_Psi2}, we see that
the smoothness of $\Phi$ is inherited by $\Psi_{k}$ (except at $t_{1}=0$) and $\Psi_{l}$ (except when $t_{2}=0$).
In these cases, though, smoothness follows from \eqref{firstFromLemma} and \eqref{secondFromLemma}.  For example,
one may compute for any $t_{2},\mu $%
\begin{eqnarray*}
\partial _{t_{2}}\Phi _{k}\left( 0,t_{2};\mu \right)
&=&\lim_{h\rightarrow 0}\frac{\Phi _{k}\left( 0,t_{2}+h;\mu \right)
-\Phi _{k}\left( 0,t_{2};\mu \right) }{h} \\
&=&0,
\end{eqnarray*}%
since both terms in the numerator vanish by Lemma \ref{smoothness_of_Phi_lemma}; thus, 
\begin{equation*}
\lim_{t_{2}\rightarrow 0}\partial _{t_{2}}\Phi _{k}\left( 0,t_{2};\mu
\right) =\partial _{t_{2}}\Phi _{k}\left( 0,0;\mu \right) =0,
\end{equation*}%
and we have regularity near zero.  We omit further details.
\end{proof}

We now will apply the implicit function theorem to the problem $\Psi\left(t_1,t_2; \mu\right) = 0$ (which, by Lemma \ref{lem:relnPhiPsi}, is equivalent to solving the bifurcation equations in the non-resonant case). For the resonant case, we define an alternate (yet still equivalent) system $\Psi^\prime = 0$ using a polar coordinates parametrization similar to that used in \cite{EhrnstromEscherWahlen}.  The following theorem is 
similar to that which appears at the end of \cite{baldiToland1}.

\begin{theorem}\label{thm:solveBifurcEqs}
Define the polynomials $p$ and $R$ as in (\ref{eq:plkDef}) and (\ref{eq:RkDef}). Suppose $k < l$ are a pair of positive integers that satisfy the following criteria for some $\tau_1^*>0$: \\
\begin{enumerate}[label={\upshape(\roman*)}]
\item $p(l,k;\tau_1^*)=0$ (2D null-space condition), and \\

\item $R(k;\tau_1^*) > 0$ (non-degeneracy condition). \\
\end{enumerate}
\begin{description}
\item[Non-resonant case]If $l \mathbin{/} k \notin \mathbb{N}$, then there exist neighborhoods $\mathcal{N}_{t}, \,\mathcal{N}_{\mu}\subseteq\mathbb{R}^2$ of (respectively) $(0,0)$ and $\mu^*$, and a smooth function $\overline{\mu}:\mathcal{N}_{t} \rightarrow \mathcal{N}_{\mu}$ that satisfies $\overline{\mu} \left(0,0\right) = \mu^\ast$ and
\begin{equation*}
\Phi\left(t_1, t_2; \overline{\mu} \left(t_1, t_2\right)\right) = 0
\end{equation*}
for all $\left(t_1, t_2\right) \in \mathcal{N}_{t}$. 
The map $\overline{\mu}$ is unique in the sense that if $\Phi\left(t_1,t_2; \mu \right) = 0$ for $\left(t_1,t_2\right) \in \mathcal{N}_t$ with $t_1 \neq 0$ and $t_2 \neq 0$, then $\mu = \overline{\mu}\left(t_1,t_2\right)$.

\item[Resonant case]If $l \mathbin{/} k \in \mathbb{N}$, then, given $\delta > 0$, there exist neighborhoods $\mathcal{N}_{r}\subseteq \mathbb{R^+}$ and $\mathcal{N}_{\mu}\subseteq\mathbb{R}^2$ of (respectively) $0$ and $\mu^*$, and a smooth function $\overline{\mu}:\mathcal{N}_{r} \times \left( \left(\delta ,\pi -\delta \right) \cup \left(-\pi +\delta , -\delta\right)\right) \rightarrow \mathcal{N}_{\mu}$ that satisfies $\overline{\mu} \left(0,\beta \right) = \mu^\ast$ and
\begin{equation*}
\Phi\left(r \cos \beta, r \sin \beta; \overline{\mu} \left(r, \beta\right)\right) = 0
\end{equation*}
for all $r \in \mathcal{N}_{r}$ and for all $\beta$ satisfying $\delta <\left\vert \beta \right\vert <\pi -\delta $.
The map $\overline{\mu}$ is unique in the sense that if $\Phi\left(r \cos \beta, r \sin \beta; \mu \right) = 0$ for nontrivial $r \in \mathcal{N}_r$ and $\beta \in \left(\delta ,\pi -\delta \right) \cup \left(-\pi +\delta, -\delta \right) \setminus \left\{\pi/2,-\pi/2\right\}$, then $\mu = \overline{\mu}\left(r,\beta \right)$.
\end{description}
\end{theorem}

\begin{proof}
We begin with the non-resonant case.  Noting condition (i), we may define $\Psi$, $\Psi_k$, and $\Psi_l$ as in (\ref{defPsi}), (\ref{defPsi1}), and (\ref{defPsi2}), and the mapping $\Psi$ is smooth by an application of Lemma \ref{smoothness_of_Psi_lemma}. We wish to apply the implicit function theorem to the problem
\begin{equation}
\label{eq:Psi1}
\begin{cases}
0 = \Psi_k \left(t_1,t_2,; \mu\right), \\
0 = \Psi_l \left(t_1,t_2,; \mu\right),
\end{cases}
\end{equation}
which, of course, means that we must look for conditions in which the matrix%
\begin{equation}
\begin{bmatrix}
\partial _{c}\Psi _{k}\left( 0,0;\mu ^{\ast }\right) & \partial _{\tau_1}\Psi
_{k}\left( 0,0;\mu ^{\ast }\right) \\ 
\partial _{c}\Psi _{l}\left( 0,0;\mu ^{\ast }\right) & \partial _{\tau_1}\Psi
_{l}\left( 0,0;\mu ^{\ast }\right)%
\end{bmatrix}%
\label{Psi_matrix}
\end{equation}%
is nonsingular.  The matrix (\ref{Psi_matrix}) is clearly equal to%
\begin{equation}
\begin{bmatrix}
\partial _{t_{1},c}^{2}\Phi _{k}\left( 0,0;\mu ^{\ast }\right) & \partial
_{t_{1},\tau_1}^{2}\Phi _{k}\left( 0,0;\mu ^{\ast }\right) \\ 
\partial _{t_{2},c}^{2}\Phi _{l}\left( 0,0;\mu ^{\ast }\right) & \partial
_{t_{2},\tau_1}^{2}\Phi _{l}\left( 0,0;\mu ^{\ast }\right)%
\end{bmatrix}
\label{Phi_matrix}
\end{equation}%
by the integral definitions (\ref{defPsi1}), (\ref{defPsi2})---we integrate an
expression with no $x$ dependence on $[0,1]$; for example,%
\begin{eqnarray*}
\partial _{c}\Psi _{k}\left( 0,0,\mu ^{\ast }\right)
&=&\int\nolimits_{0}^{1}\partial _{t_{1},c}^{2}\Phi _{k}\left( 0,0;\mu^\ast
\right) dx, \\
&=&\partial _{t_{1},c}^{2}\Phi _{k}\left( 0,0;\mu^\ast \right).
\end{eqnarray*}

Explicitly, we are able to calculate the derivatives that appear in (\ref{Phi_matrix}).  Indeed, we may calculate (in the appropriate neighborhoods provided by Lemma \ref{baldi-toland-like-lemma})
\begin{eqnarray*}
\partial _{t_{1}}\Phi _{k}\left( 0,0;\mu \right) &=&\Pi _{k}L\left( \mu \right) \left[ \frac{\partial v}{\partial _{t_{1}}}+D%
\overline{y}\left( \mu \right) \frac{\partial v}{\partial _{t_{1}}}\right], \\
&=&\Pi _{k}L\left( \mu \right) \frac{\partial v}{\partial _{t_{1}}}
\end{eqnarray*}%
(note that $D\overline{y}\left( \mu \right) =0$ by Lemma \ref{baldi-toland-like-lemma}). Similarly,%
\begin{equation*}
\partial _{t_{2}}\Phi _{l}\left( 0,0;\mu \right) =\Pi _{l}L\left( \mu
\right) \frac{\partial v}{\partial _{t_{2}}}.
\end{equation*}  Then,
\begin{eqnarray*}
\partial _{t_{1},c}^{2}\Phi _{k}\left( 0,0;\mu \right)  &=& \Pi _{k} \partial_{c} L\left( \mu \right) \frac{\partial v}{\partial _{t_{1}}}, \\
\partial _{t_{1},\tau_1}^{2}\Phi _{k}\left( 0,0;\mu \right)  &=&  \Pi _{k} \partial_{\tau_1} L\left( \mu \right) \frac{\partial v}{\partial _{t_{1}}},
\end{eqnarray*}
and likewise for $\partial_{t_2} \Phi_l.$ Thus, (\ref{Phi_matrix}) becomes 
\begin{equation}
\begin{bmatrix}
\Pi _{k} \partial_{c} L\left( \mu^\ast \right) \frac{\partial v}{\partial _{t_{1}}} & \Pi _{k} \partial_{\tau_1} L\left( \mu^\ast \right) \frac{\partial v}{\partial _{t_{1}}} \\ 
\Pi _{l} \partial_{c} L\left( \mu^\ast \right) \frac{\partial v}{\partial _{t_{2}}} & \Pi _{l} \partial_{\tau_1} L\left( \mu^\ast \right) \frac{\partial v}{\partial _{t_{2}}}
\end{bmatrix}.
\label{projections_matrix}
\end{equation}
We may then explicitly calculate the entries of (\ref{projections_matrix}), and subsequently the determinant
\begin{eqnarray}
\det \begin{bmatrix}
\Pi _{k} \partial_{c} L\left( \mu^\ast \right) \frac{\partial v}{\partial _{t_{1}}} & \Pi _{k} \partial_{\tau_1} L\left( \mu^\ast \right) \frac{\partial v}{\partial _{t_{1}}} \\ 
\Pi _{l} \partial_{c} L\left( \mu^\ast \right) \frac{\partial v}{\partial _{t_{2}}} & \Pi _{l} \partial_{\tau_1} L\left( \mu^\ast \right) \frac{\partial v}{\partial _{t_{2}}}
\end{bmatrix}\label{eq:det_calculated}\\
=\pm \frac{\left( k-l\right) S \sqrt{2\pi R\left( k;\tau _{1}^{\ast
}\right) }}{\left( c_{\pm } \left (k; \tau_1^\ast \right)\right)^2 \left( a_{k}^{2}+d^{2}\right) \left(
a_{l}^{2}+d^{2}\right) \sqrt{k M}}, \notag
\end{eqnarray}%
where $R\left( k; \tau_1 \right)$ is as defined in (\ref{eq:RkDef}),
\begin{equation*}
a_j := -\frac{n\left( j;\mu^\ast \right) }{2\pi ^{2}j},
\end{equation*} (with $n\left( j;\mu \right)$ as defined in (\ref{def:nFromAdjointEigenfcns})) and
\begin{eqnarray*}
d &:=&  -c_\pm \left (k; \tau_1^\ast \right)M+A\overline{\gamma }\pi,  \\
&=& \mp \sqrt{\frac{R\left(k; \tau_1^\ast \right)}{2 k M \pi}}.
\end{eqnarray*}Since $k \neq l$ by assumption and $R\left(
k ; \tau_1^\ast \right) > 0$ by condition (ii) (this non-degeneracy condition is analogous to equation (8.5) in \cite{baldiToland1}), the determinant (\ref{eq:det_calculated}) is nonzero.

Note
\begin{eqnarray*}
\Psi \left( 0,0;\mu ^{\ast }\right) &=&
\begin{bmatrix}
\int\nolimits_{0}^{1}\partial _{t_{1}}\Phi _{k}\left( 0,0;\mu ^{\ast
}\right) dx \\ 
\int\nolimits_{0}^{1}\partial _{t_{2}}\Phi _{l}\left( 0,0;\mu ^{\ast
}\right) dx
\end{bmatrix},
\\
&=&
\begin{bmatrix}
\Pi _{k}L\left( \mu ^{\ast }\right) \frac{\partial v}{\partial _{t_{1}}} \\ 
\Pi _{l}L\left( \mu ^{\ast }\right) \frac{\partial v}{\partial _{t_{2}}}%
\end{bmatrix},
\\
&=&0,
\end{eqnarray*}%
given that $L\left(\mu^\ast \right)$ has a (double) zero eigenvalue corresponding to eigenfunctions $\partial v/ \partial _{t_{1}},\, %
\partial v / \partial _{t_{2}}$. Thus,  we can apply the implicit function theorem to (\ref{eq:Psi1}) and produce such neighborhoods $\mathcal{N}_{t}, \,\mathcal{N}_{\mu}$ and the function $\overline{\mu}$ for which
\begin{equation*}
\Psi\left(t_1, t_2; \overline{\mu} \left(t_1, t_2\right)\right) = 0.
\end{equation*}
Then, by applying Lemma \ref{lem:relnPhiPsi}, we have the desired result for $\Phi$.  

Next, we approach the resonant case.  From Lemma \ref{smoothness_of_Phi_lemma}, we still have%
\begin{equation*}
\Phi _{k}\left( 0,t_{2};\mu \right) =0\text{ \ \ for all }t_{2},\mu \text{,}
\end{equation*}%
yet $\Phi _{l}\left( t_{1},0;\mu \right)$ might not necessarily vanish.  To proceed, we define $\Psi _{k}$ and $\Psi
_{l}$ differently, this time using polar coordinates $\left(
t_{1},t_{2}\right) =\left( r\cos \beta ,r\sin \beta \right) $. \ Put%
\begin{eqnarray}
\Psi _{k}^\prime\left( r,\beta ;\mu \right) &:=&\int\nolimits_{0}^{1}\partial
_{t_{1}}\Phi _{k}\left( xr\cos \beta ,r\sin \beta ;\mu \right) dx, \label{defPsi1prime}\\
\Psi _{l}^\prime \left( r,\beta ;\mu \right) &:=&\int\nolimits_{0}^{1}\left[
\partial _{t_{1}}\Phi _{l}\left( xr\cos \beta ,xr\sin \beta ;\mu \right)
\cos \beta \right. \label{defPsi2prime}\\
&&\;\;\;+\left. \partial _{t_{2}}\Phi _{l}\left( xr\cos \beta ,xr\sin \beta ;\mu
\right) \sin \beta \right] dx. \notag
\end{eqnarray}%
As in (\ref{relationship_between_Phi_Psi1}), we still have the relation
\begin{equation*}
\Psi _{k}^\prime \left( r,\beta ;\mu \right) =\frac{1}{r\cos \beta }\Phi _{k}\left(
r\cos \beta ,r\sin \beta ;\mu \right), \;\;\; r \cos \beta \neq 0;
\end{equation*}%
yet,
\begin{equation*}
\Psi _{l}^\prime \left( r,\beta ;\mu \right) =\frac{1}{r}\Phi _{l}\left( r\cos \beta
,r\sin \beta ;\mu \right), \;\;\; r \neq 0,
\end{equation*}
(each of which can be shown by a simple substitution as in the proof of Lemma \ref{lem:relnPhiPsi}).  

Thus, in lieu of Lemma \ref{lem:relnPhiPsi}, we have equivalence of the systems
\begin{equation*}
\begin{cases}
\Phi _{k} \left( r\cos \beta ,r\sin \beta ;\mu \right) =0 \\ 
\Phi _{l} \left( r\cos \beta ,r\sin \beta ;\mu \right) =0%
\end{cases}
\end{equation*}%
and
\begin{equation*}
\begin{cases}
\Psi _{k}^\prime\left( r,\beta ;\mu \right) =0 \\ 
\Psi _{l}^\prime\left( r,\beta ;\mu \right) =0%
\end{cases}
\end{equation*}%
whenever $r\cos \beta \neq 0$.  Note that $r=0$ corresponds to the trivial
solution, while $\cos \beta =0$ corresponds to a pure $l$-wave.  Thus, we
concern ourselves with solving $\Psi _{k}^\prime \left( r,\beta ;\mu \right) =\Psi
_{l}^\prime\left( r,\beta ;\mu \right) =0$, and (like in the non-resonant case) appeal to the implicit function theorem
to do so.

Note that $\Psi _{k}^\prime\left( 0,\beta ;\mu \right) =\Psi _{l}^\prime\left(
0,\beta ;\mu \right) =0$ for all $\beta $ and for all $\mu $, and thus we also still have smoothness at $r=0$.  By the integral definitions (\ref{defPsi1prime}), (\ref{defPsi2prime}) and by our work earlier for the non-resonant case, the matrix
\begin{equation}
\label{Psi_matrix_resonant}
\begin{bmatrix}
\partial _{c}\Psi _{k}^\prime \left( 0,\beta ;\mu ^{\ast }\right) & \partial _{\tau
_{1}}\Psi _{k}^\prime\left( 0,\beta ;\mu ^{\ast }\right) \\ 
\partial _{c}\Psi _{l}^\prime\left( 0,\beta ;\mu ^{\ast }\right) & \partial _{\tau
_{1}}\Psi _{l}^\prime\left( 0,\beta ;\mu ^{\ast }\right)%
\end{bmatrix}%
\end{equation}%
is equal to%
\begin{equation*}
\begin{bmatrix}
\Pi _{k}\partial _{c}L\left( \mu ^{\ast }\right) \frac{\partial v}{\partial
t_{1}} & \sim \;\;\; \\ 
\left( \cos \beta \right) \Pi _{l}  \partial _{c}L\left( \mu ^{\ast }\right) 
\frac{\partial v}{\partial t_{1}}+\left( \sin \beta \right) \Pi _{l}\partial
_{c}L\left( \mu ^{\ast }\right) \frac{\partial v}{\partial t_{2}} & \sim \;\;\;%
\end{bmatrix},%
\end{equation*}where the ``$\sim$" in the second column indicate the analogous expression with $\partial_{\tau_1}$ replacing $\partial_c$.  Examining the second row, we see that%
\begin{equation*}
\Pi _{l}\partial _{c}L\left( \mu ^{\ast }\right) \frac{\partial v}{\partial
t_{1}} = \frac{\left\langle \partial _{c}L\left( \mu ^{\ast }\right) \frac{%
\partial v}{\partial t_{1}},\phi _{l}\left( \cdot ;\mu ^{\ast }\right)
\right\rangle }{\left\langle \phi _{l}\left( \cdot ;\mu ^{\ast }\right)
,\phi _{l}\left( \cdot ;\mu ^{\ast }\right) \right\rangle } 
=0,
\end{equation*}%
and likewise%
\begin{equation*}
\Pi _{l}\partial _{\tau _{1}}L\left( \mu ^{\ast }\right) \frac{\partial v}{%
\partial t_{1}}=0,
\end{equation*}%
since $\sin k\alpha $ and $\sin l\alpha $ are orthogonal to each other with
respect to the $L^{2}$ inner product, as are $\cos k\alpha $ and $\cos
l\alpha $. \ Thus, the matrix (\ref{Psi_matrix_resonant}) is equal to%
\begin{equation*}
\begin{bmatrix}
\Pi _{k}\partial _{c}L\left( \mu ^{\ast }\right) \frac{\partial v}{\partial
t_{1}} & \Pi _{k}\partial _{\tau _{1}}L\left( \mu ^{\ast }\right) \frac{%
\partial v}{\partial t_{1}} \\ 
\left( \sin \beta \right) \Pi _{l} \partial _{c}L\left( \mu ^{\ast }\right) 
\frac{\partial v}{\partial t_{2}}& \left( \sin \beta \right) \Pi
_{l}\partial _{\tau _{1}}L\left( \mu ^{\ast }\right) \frac{\partial v}{%
\partial t_{2}}%
\end{bmatrix},
\end{equation*}%
and hence%
\begin{equation*}
\det 
\begin{bmatrix}
\partial _{c}\Psi _{k}^\prime\left( 0,\beta ;\mu ^{\ast }\right) & \partial _{\tau
_{1}}\Psi _{k}^\prime\left( 0,\beta ;\mu ^{\ast }\right) \\ 
\partial _{c}\Psi _{l}^\prime\left( 0,\beta ;\mu ^{\ast }\right) & \partial _{\tau
_{1}}\Psi _{l}^\prime\left( 0,\beta ;\mu ^{\ast }\right)%
\end{bmatrix}%
=\sin \beta \det 
\begin{bmatrix}
\Pi _{k}\partial _{c}L\left( \mu ^{\ast }\right) \frac{\partial v}{\partial
t_{1}} & \Pi _{k}\partial _{\tau _{1}}L\left( \mu ^{\ast }\right) \frac{%
\partial v}{\partial t_{1}} \\ 
\Pi _{l}\partial _{c}L\left( \mu ^{\ast }\right) \frac{\partial v}{\partial
t_{2}} & \Pi _{l}\partial _{\tau _{1}}L\left( \mu ^{\ast }\right) \frac{%
\partial v}{\partial t_{2}}%
\end{bmatrix}%
.
\end{equation*}%
The determinant on the right-hand side was calculated earlier; see (\ref{eq:det_calculated}).  Thus, provided $\sin \beta \neq 0$, the same
non-degeneracy condition $R\left( k;\tau _{1}^{\ast }\right) >0$ ensures that
the implicit function theorem may be applied as in the non-resonant case.  As in \cite{EhrnstromEscherWahlen}, by fixing $\delta$ and considering the compact set $\left[\delta ,\pi -\delta \right] \cup \left[-\pi +\delta, -\delta \right]$, we may produce such a neighborhood $\mathcal{N}_r$ and solution $\overline{\mu}$ that is valid over the entire set $\mathcal{N}_r \times \left(\delta ,\pi -\delta \right) \cup \left( -\pi +\delta,-\delta \right)$.
\end{proof}

We are now ready to prove Theorem \ref{thm:main}; the technical work is largely complete.

\subsubsection{Proof of main theorem}
Decompose the problem $F\left(\theta, \gamma_1; \mu \right) = 0$ via the Lyapunov--Schmidt process detailed in Section \ref{subsubsec:lyapunovSchmidt}.  Theorem \ref{thm:solveBifurcEqs}, which concerns solutions to the bifurcation equation $\Pi_{\mathcal{V}^\dagger} F\left(\theta, \gamma_1; \mu \right) = 0$, provides us with the neighborhoods $\mathcal{N}_t, \mathcal{N_{\mu}}$ and the map $\overline{\mu}$ in the non-resonant case. In the resonant case, we obtain the corresponding neighborhoods and solutions for any choice of $\delta > 0$.  In either case, we can subsequently apply Lemma \ref{baldi-toland-like-lemma} to produce neighborhoods $\mathcal{N}_\mathcal{V},\,\mathcal{N}_{\mathcal{R}^\dagger}$ and the map $\overline{y}$. \qed

\subsubsection{Additional remarks} Note that Theorem \ref{thm:main} provides a two-dimensional sheet of solutions over which both $c$ and $\tau_1$ may vary.  However, the theorem does not cover the existence of a specific curve of solutions in which $\tau_1 = \tau_1^\ast$ is fixed over the entire curve.  As discussed in the introduction, such resonant $\left(k,l\right)$ combination wave solutions (i.e.~Wilton ripples) are of interest.  In the next section, we establish some asymptotic results regarding the specific case $\left(k,l\right) = \left(1,2\right)$ of Wilton ripples.  The results of these asymptotics are subsequently used to provide initial guesses for the numerical computations in Section \ref{sec:numerical}.

\section{Wilton ripples asymptotics}
\label{sec:wiltonRipples}

We return to an earlier form of the traveling wave equations (\ref{eq:gammaEqBeforeSigmaTauDivision}) and (\ref{eq:thetaEq1}) and proceed to apply typical perturbation theory techniques, in order to investigate asymptotics of possible $\left(1,2\right)$ Wilton ripple solutions. Here, we make no claim on the convergence of formal series representations of solutions; such a proof (in the vein of \cite{reederShinbrot2}) will be the subject of future study.

Throughout this section, we assume that $M=2\pi $ and that $\gamma $ has zero
mean (i.e. $\overline{\gamma }=0$). We thus assume a solution of the form%
\begin{eqnarray}
\theta &=&\varepsilon \theta _{1}+\varepsilon ^{2}\theta _{2}+O\left(
\varepsilon ^{3}\right), \label{eq:theta_expansion} \\
\gamma &=&\varepsilon \gamma _{1}+\varepsilon ^{2}\gamma _{2}+O\left(
\varepsilon ^{3}\right) , \label{eq:gamma_expansion}
\end{eqnarray}%
with wave speed%
\begin{equation}
c=c_{0}+\varepsilon c_{1} + O\left(\varepsilon^2 \right), \label{eq:c_expansion}
\end{equation}
where $\varepsilon$ is taken to be small.  We will usually only display the number of terms necessary
in $\varepsilon $ expansions to ensure the equations (\ref{eq:gammaEqBeforeSigmaTauDivision}), (\ref{eq:thetaEq1}) hold up to $%
O\left( \varepsilon ^{2}\right) $. For additional brevity, we present only the final results of the calculations (more detail can be found in \cite{SulonThesis}).

\subsection{Expansion of equations}\label{subsec:expansionEqs}
Substituting the expansions (\ref{eq:theta_expansion}), (\ref{eq:gamma_expansion}), and (\ref{eq:c_expansion}) into (\ref{eq:gammaEqBeforeSigmaTauDivision}) and simplifying (with special care taken with expanding the Birkhoff-Rott integral (\ref{eq:birkhoffRottDef})), we obtain
\begin{eqnarray}
\;\;\;\;\;\;\;\; 0 &=&-S\left( \varepsilon \partial _{\alpha }^{4}\theta _{1}-\varepsilon
\tau _{1}\partial _{\alpha }^{2}\theta _{1}+\varepsilon ^{2}\partial
_{\alpha }^{4}\theta _{2}-\varepsilon ^{2}\tau _{1}\partial _{\alpha
}^{2}\theta _{2}\right) +\varepsilon ^{2}\widetilde{A}\partial _{\alpha
}\left( \theta _{1}^{2}\right) \label{eq:approx1} \\
&&+\varepsilon c_{0}\partial _{\alpha }\gamma _{1}+\varepsilon
^{2}c_{1}\partial _{\alpha }\gamma _{1}+\varepsilon ^{2}c_{0}\partial
_{\alpha }\gamma _{2} \notag \\
&&-A\left( \varepsilon \left(2\theta _{1}\right)+\frac{\varepsilon ^{2}\partial _{\alpha
}\left( \gamma _{1}^{2}\right) }{4}+\varepsilon ^{2}\left(2\theta _{2}\right)\right) \notag \\
&&-A\left(\varepsilon
^{2}c_{0}\partial _{\alpha }\left\{ \left[ -c_{0}\theta _{1}^{2}-\theta
_{1}H\gamma _{1}-\left( \partial _{\alpha }^{-1}\theta _{1}\partial _{\alpha
}H\gamma _{1}-\partial _{\alpha }H\left( \gamma _{1}\partial _{\alpha
}^{-1}\theta _{1}\right) \right) \right] \right\} \right) \notag \\
&&+O\left(\varepsilon^3\right). \notag
\end{eqnarray}%
Similarly, the second equation (\ref{eq:thetaEq1}) becomes%
\begin{equation}
0 = \varepsilon c_{0}\theta _{1}+\frac{\varepsilon }{2}H\gamma
_{1}+\varepsilon ^{2}c_{0}\theta _{2}+\varepsilon ^{2}c_{1}\theta _{1}+\frac{%
\varepsilon ^{2}}{2}H\gamma _{2} +O\left(\varepsilon^3\right).\label{eq:approx2}
\end{equation}%

\subsection{Linearized equations}

Truncating our expansions (\ref{eq:approx1}), (\ref{eq:approx2}) up to and including $O\left(
\varepsilon \right) $, we write the linearization of (\ref{eq:gammaEqBeforeSigmaTauDivision}), (\ref{eq:thetaEq1}) as
\begin{eqnarray*}
-S\left( \partial _{\alpha }^{4}\theta _{1}-\tau _{1}\partial _{\alpha
}^{2}\theta _{1}\right) +c_{0}\partial _{\alpha }\gamma _{1}-2A\theta
_{1} &=&0, \\
c_{0}\theta _{1}+\frac{1}{2}H\gamma _{1} &=&0,
\end{eqnarray*}%
or, in matrix form,%
\begin{equation*}
\mathcal{L}%
\begin{bmatrix}
\theta _{1} \\ 
\gamma _{1}%
\end{bmatrix}%
=0,
\end{equation*}%
where%
\begin{equation*}
\mathcal{L}:=%
\begin{bmatrix}
-S\left( \partial _{\alpha }^{4}-\tau _{1}\partial _{\alpha }^{2}\right) -2A
& c_{0}\partial _{\alpha } \\ 
c_{0} & \frac{1}{2}H%
\end{bmatrix}%
.
\end{equation*}%
Using the appropriate Fourier symbols (let $k\in 
\mathbb{Z}
$ denote the frequency domain variable, as in Section \ref{subsec:linearization}), and setting $\det \mathcal{L} = 0$, we obtain the linear wave speed%
\begin{equation*}
c_{0} =\pm \sqrt{\frac{S\left\vert k\right\vert ^{3}}{2}+\frac{S\tau
_{1}\left\vert k\right\vert }{2}+\frac{A}{\left\vert k\right\vert }}.
\end{equation*}

\begin{remark}
Note that $\mathcal{L}$ here denotes the linearization of (\ref{eq:gammaEqBeforeSigmaTauDivision}), (\ref{eq:thetaEq1}), while $L\left(\mu\right)$ (presented in Section \ref{subsec:linearization}) is the linearization of the ``identity plus compact" reformulation $\left( \theta -\Theta \left( \theta ,\gamma _{1};\mu \right) ,\gamma_{1}-\Gamma \left( \theta ,\gamma _{1};\mu \right) \right) = (0,0)$ of these equations. However, as expected, the above expression for the linear wave speed $c_0$ coincides with our earlier presented $c_\pm \left(k; \tau_1 \right)$ (see (\ref{c_pm})) in the case $M = 2 \pi$, $\overline{\gamma} = 0$.
\end{remark}

Since we are interested in the $(1,2)$ case of Wilton ripples, we assume that the leading order terms of $\left(\theta,\gamma\right)$ are the following linear combination of eigenfunctions%
\begin{equation*}
\begin{bmatrix}
\theta _{1} \\ 
\gamma _{1}%
\end{bmatrix}%
=%
\begin{bmatrix}
i \\ 
2c_{0}%
\end{bmatrix}%
\exp \left( i\alpha \right) +%
\begin{bmatrix}
-i \\ 
2c_{0}%
\end{bmatrix}%
\exp \left( -i\alpha \right) + t_2 \left( 
\begin{bmatrix}
i \\ 
2c_{0}%
\end{bmatrix}%
\exp \left( 2i\alpha \right) +%
\begin{bmatrix}
-i \\ 
2c_{0}%
\end{bmatrix}%
\exp \left( -2i\alpha \right) \right),
\end{equation*}
where $t_2 \in \mathbb{C}$ is yet to be determined.  

\subsection{Second-order equations}
The $O\left(\varepsilon ^{2}\right) $ terms of (\ref{eq:approx1}), (\ref{eq:approx2}) yield the equations
\begin{eqnarray*}
&&-S\left( \partial _{\alpha }^{4}\theta _{2}-\tau _{1}\partial _{\alpha
}^{2}\theta _{2}\right) -2A\theta _{2}+c_{0}\partial _{\alpha }\gamma _{2} \\
&=&-\widetilde{A}\partial _{\alpha }\left( \theta _{1}^{2}\right)
-c_{1}\partial _{\alpha }\gamma _{1} \\
&&+A\left( \frac{\partial _{\alpha }\left(
\gamma _{1}^{2}\right) }{4}-c_{0}\partial _{\alpha }\left\{ c_{0}\theta
_{1}^{2}+\theta _{1}H\gamma _{1}+\partial _{\alpha }^{-1}\theta _{1}\partial
_{\alpha }H\gamma _{1}-\partial _{\alpha }H\left( \gamma _{1}\partial
_{\alpha }^{-1}\theta _{1}\right) \right\} \right)
\end{eqnarray*}%
\newline
and%
\begin{equation*}
c_{0}\theta _{2}+\frac{1}{2}H\gamma _{2}=-c_{1}\theta _{1},
\end{equation*}which are of the form%
\begin{equation*}
\mathcal{L}%
\begin{bmatrix}
\theta _{2} \\ 
\gamma _{2}%
\end{bmatrix}
=
\begin{bmatrix}
\text{RHS1} \\ 
\text{RHS2}%
\end{bmatrix},
\end{equation*}
where
\begin{eqnarray*}
\text{RHS1} &:=& -\widetilde{A}\partial _{\alpha }\left( \theta _{1}^{2}\right)
-c_{1}\partial _{\alpha }\gamma _{1} \\
&&+A\left( \frac{\partial _{\alpha }\left(
\gamma _{1}^{2}\right) }{4}-c_{0}\partial _{\alpha }\left\{ c_{0}\theta
_{1}^{2}+\theta _{1}H\gamma _{1}+\partial _{\alpha }^{-1}\theta _{1}\partial
_{\alpha }H\gamma _{1}-\partial _{\alpha }H\left( \gamma _{1}\partial
_{\alpha }^{-1}\theta _{1}\right) \right\} \right), \\
\text{RHS2} &:=& -c_{1}\theta _{1}.
\end{eqnarray*}

By the Fredholm alternative, we need%
\begin{equation*}
\left\langle v,%
\begin{bmatrix}
\text{RHS1} \\ 
\text{RHS2}%
\end{bmatrix}%
\right\rangle =0
\end{equation*}%
for all $v\in \func{Ker} \mathcal{L}^{\ast }$, where $\mathcal{L}^{\ast }$ denotes the Hermitian adjoint of $\mathcal{L}$. By explicitly calculating this inner product against each of the two basis elements of $\func{Ker} \mathcal{L}^{\ast }$, we obtain the equations 
\begin{eqnarray*}
-2\widetilde{A}t_2 +2A t_2  c_{0}^{2}-4c_{0}c_{1} &=&0, \\
2\widetilde{A}+4Ac_{0}^{2}-8t_2 c_{0}c_{1} &=&0.
\end{eqnarray*}%
Solving this system for $c_1$ and $t_2$, we obtain%
\begin{equation}\label{WiltonC1}
c_{1}=\pm \frac{\sqrt{\left( Ac_{0}^{2}-\widetilde{A}\right) \left(
2Ac_{0}^{2}+\widetilde{A}\right) }}{2\sqrt{2}c_{0}}
\end{equation}
and
\begin{equation*}
t_2  = \pm \sqrt{\frac{2Ac_{0}^{2}+\widetilde{A}}{2Ac_{0}^{2}-2\widetilde{A}}},
\end{equation*}
where the sign of $t_2$ is determined by that of $c_1$.  Note that by our expression for $c_{1}$, we require the stipulation that 
\begin{equation*}
\left(
Ac_{0}^{2}-\widetilde{A}\right) \left( 2Ac_{0}^{2}+\widetilde{A}\right) \geq
0.
\end{equation*}

In the next section, we compute examples of branches of traveling waves where the kernel of the linearization is two dimensional in both the resonant Wilton ripple ($l \mathbin{/} k \in \mathbb{N}$) and non-resonant Stokes' wave ($l \mathbin{/} k \notin \mathbb{N}$) cases.

\section{Numerical methods and results}\label{sec:numerical}

\begin{figure}[tp]
\centerline{\includegraphics[width=0.5\textwidth]{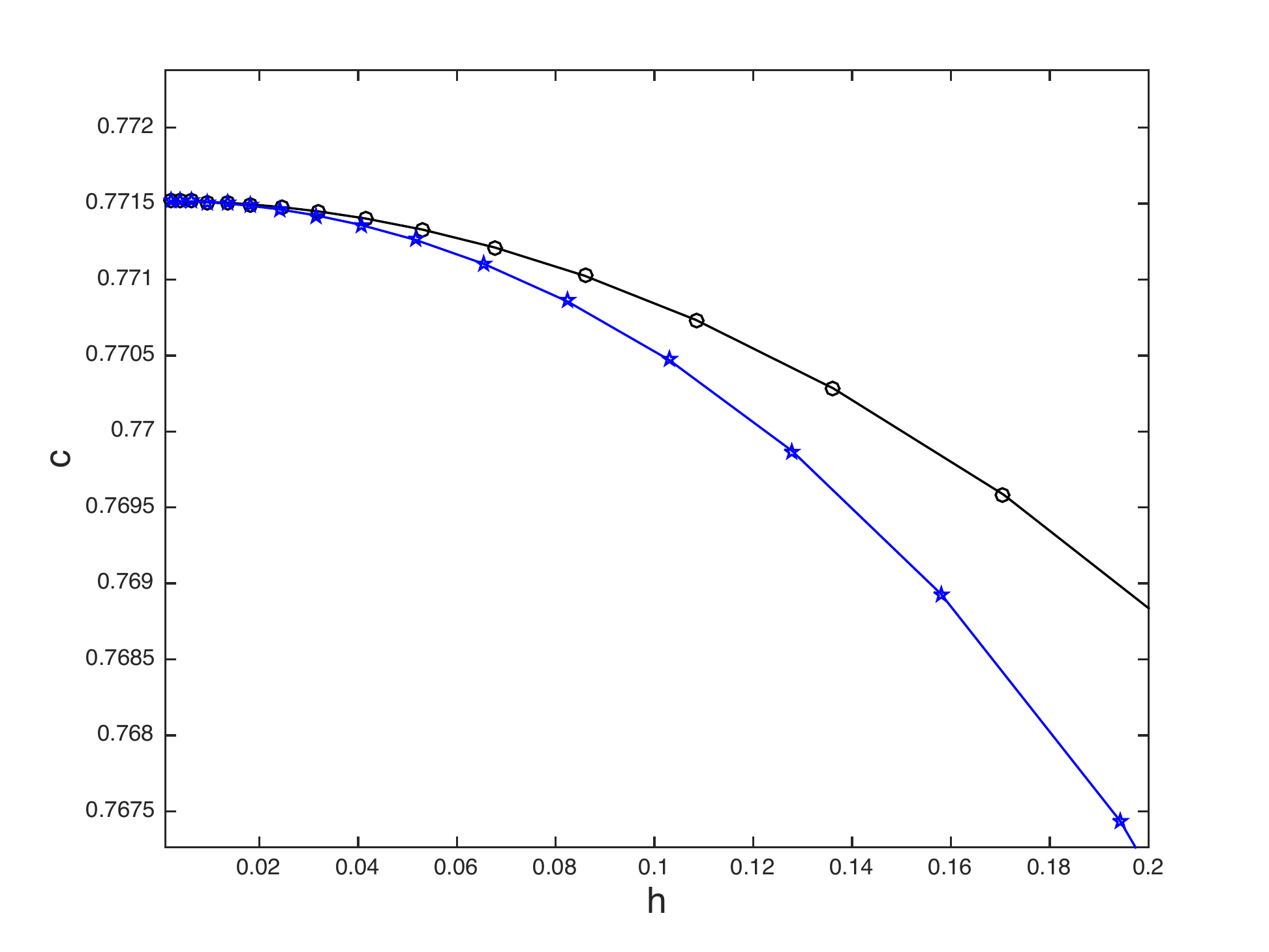}\includegraphics[width=0.5\textwidth]{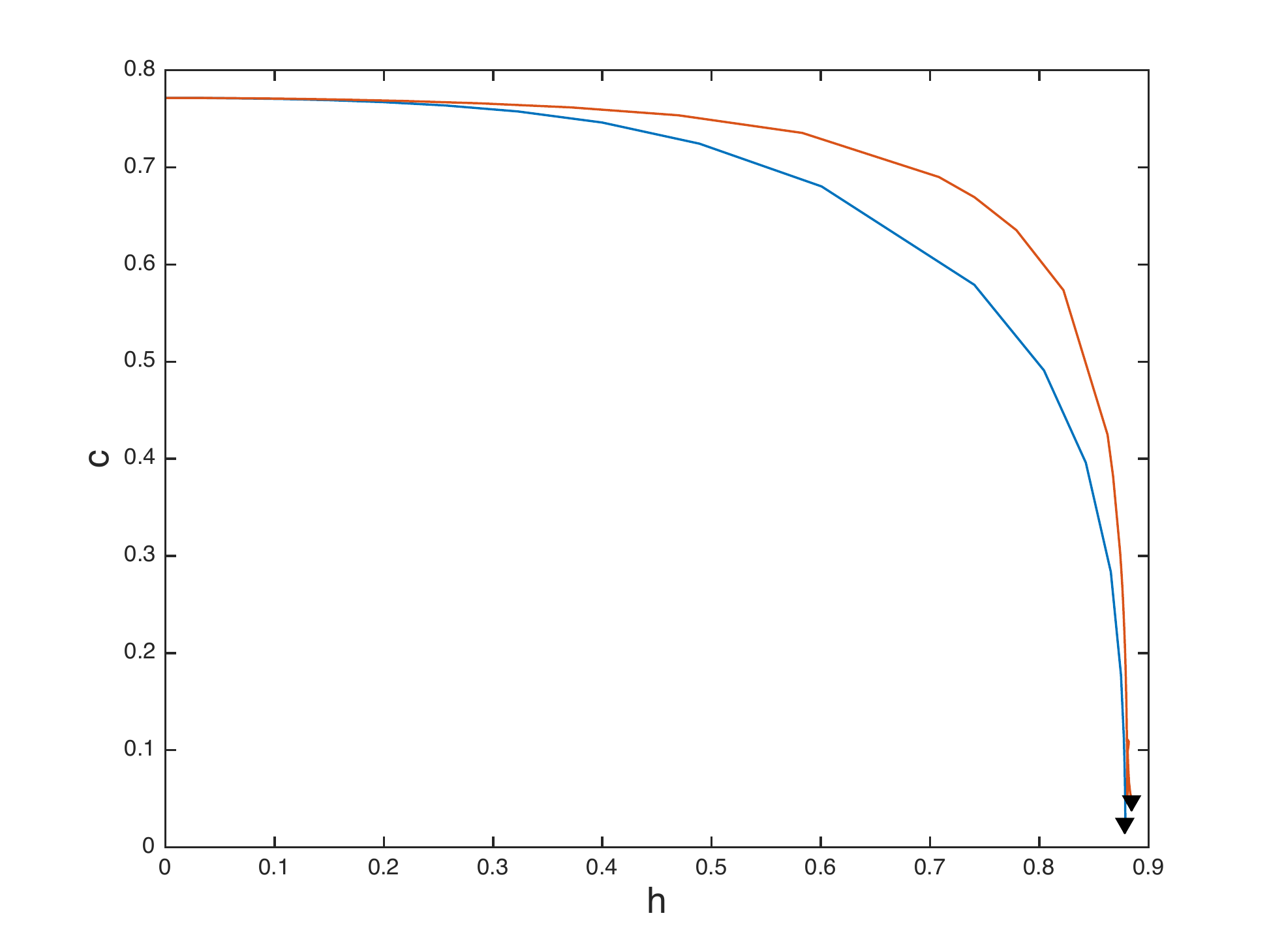}}
\centerline{\includegraphics[width=0.5\textwidth]{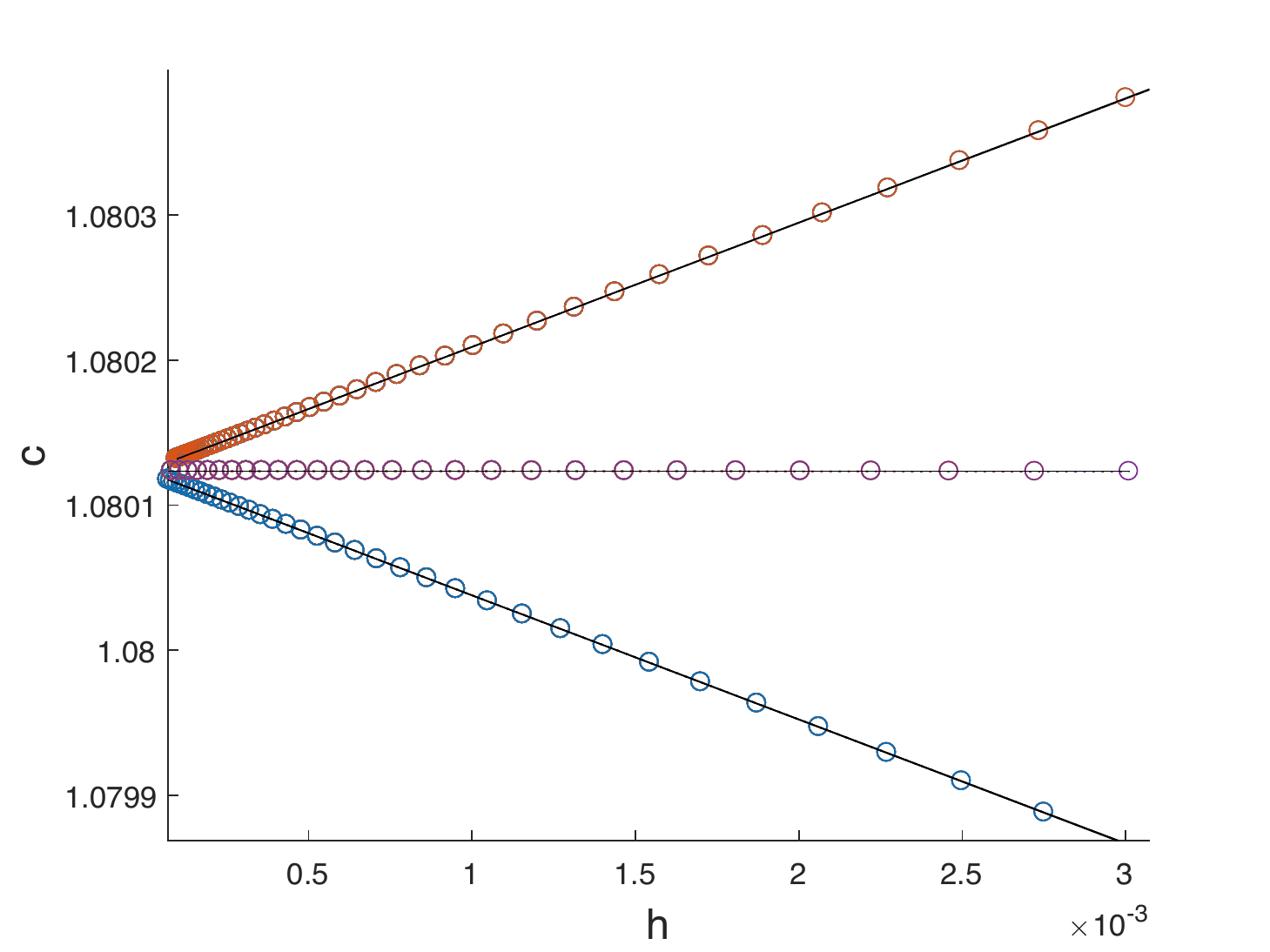}\includegraphics[width=0.5\textwidth]{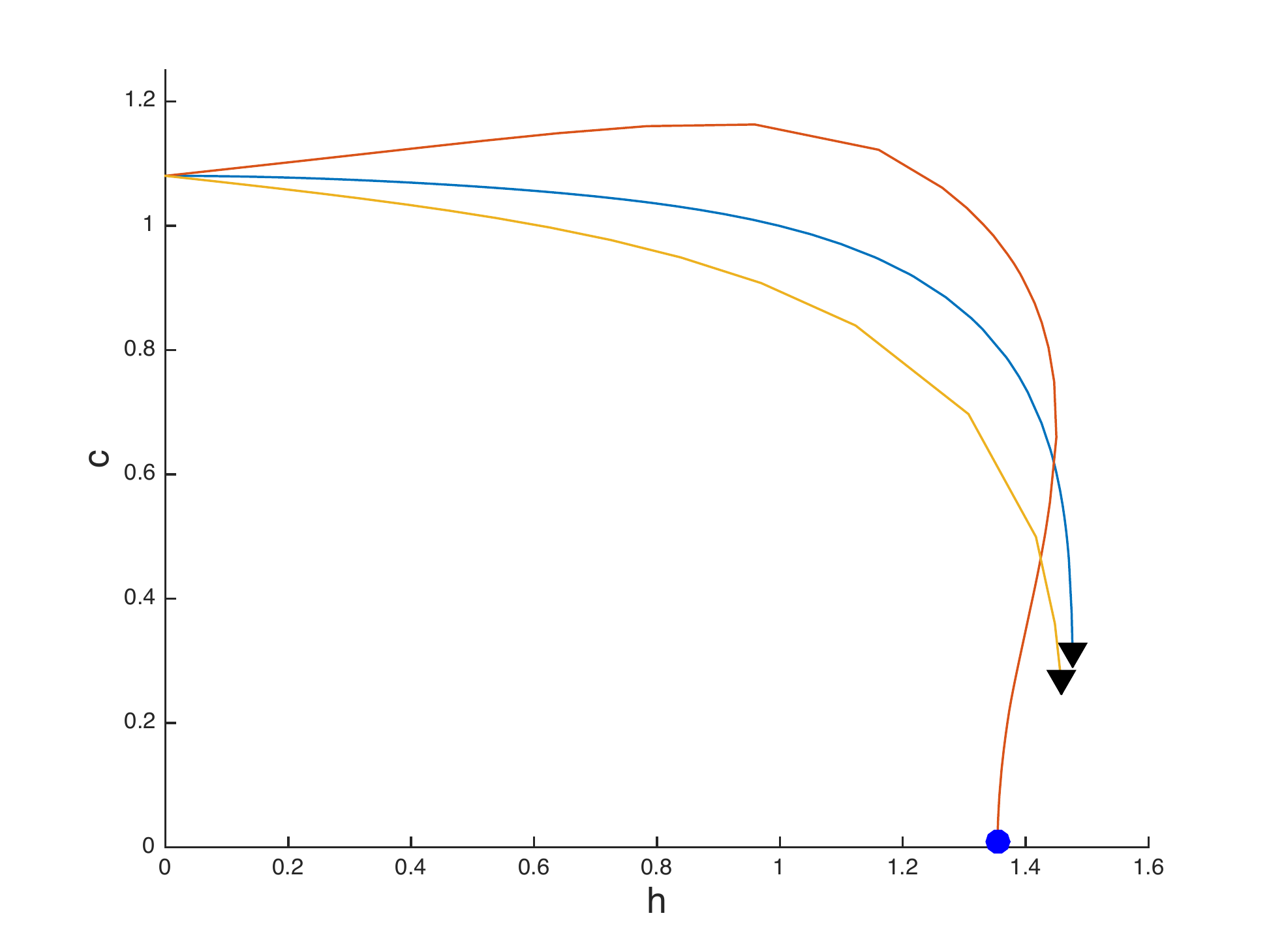}}
\caption{\it Example speed-amplitude curves of non-resonant (top row) and resonant (bottom row) branches of traveling waves.  The non-resonant configuration is at $S=1/63, \tilde{A}=1/10, A=1, \tau=2$; the resonant configuration is at $S=1/9, \tilde{A}=1/5, A=1, \tau=2$.  The left column is a close-up of small amplitude behavior; the right column portrays the global branches.  In the bottom left panel the small amplitude computations, open circles, are compared to the asymptotic predictions for Wilton ripples, solid lines.  In the right column, the extreme configurations are noted by solid markers:  triangles correspond to self-intersecting profiles (bottom left and bottom center panels in Figure \ref{Profiles} and the bottom row of Figure \ref{ProfilesNonRes}) while
the circle marks a static wave profile (the bottom right panel in Figure \ref{Profiles}).  \label{Branches}}
\end{figure}

Our numerical methods are essentially the same as those in \cite{AkersAmbroseSulonPeriodicHydroelasticWaves}, which themselves are similar to those in \cite{AkersAmbroseWrightTravelingWavesSurfaceTension}. For our computations, we continue to work with the version (\ref{eq:gammaEqBeforeSigmaTauDivision}) and (\ref{eq:thetaEq1}) of our equations. The horizontal domain has width $M = 2 \pi$ ($\alpha \in \left[- \pi, \pi\right]$), and
the functions $\theta, \gamma$ are projected  onto a finite-dimensional Fourier basis:
\begin{equation*}
\theta \left( \alpha \right) =\sum\limits_{k=-N}^{k=N}a_{k}\exp \left(
ik\alpha \right) ,\text{ \ \ \ \ \ }\gamma \left( \alpha \right)
=\sum\limits_{k=-N}^{k=N}b_{k}\exp \left( ik\alpha \right) .
\end{equation*}%
By symmetry considerations (i.e. $\theta$ is odd and $\gamma$ is even), both $a_{-k} = -a_k$ (hence $a_0 = 0$) and $b_{-k} = b_k$ (with $b_0 = \overline{\gamma}$) and thus the dimension of the system is somewhat reduced. The mean shear $\overline{\gamma}$, is specified in advance, so computing a traveling wave requires determining $2N + 1$ values: $a_{1},\dots ,a_{N},b_{1},\dots ,b_{N},$ and the wave-speed $c$.  Projecting (\ref{eq:gammaEqBeforeSigmaTauDivision}) and (\ref{eq:thetaEq1}) into Fourier space yields a system of $2N$ algebraic equations; then, to complete the system, we append another equation which specifies some measure of the solution's amplitude \cite{AkersAmbroseSulonPeriodicHydroelasticWaves}.

\begin{figure}[tp]
\centerline{\includegraphics[width=0.5\textwidth]{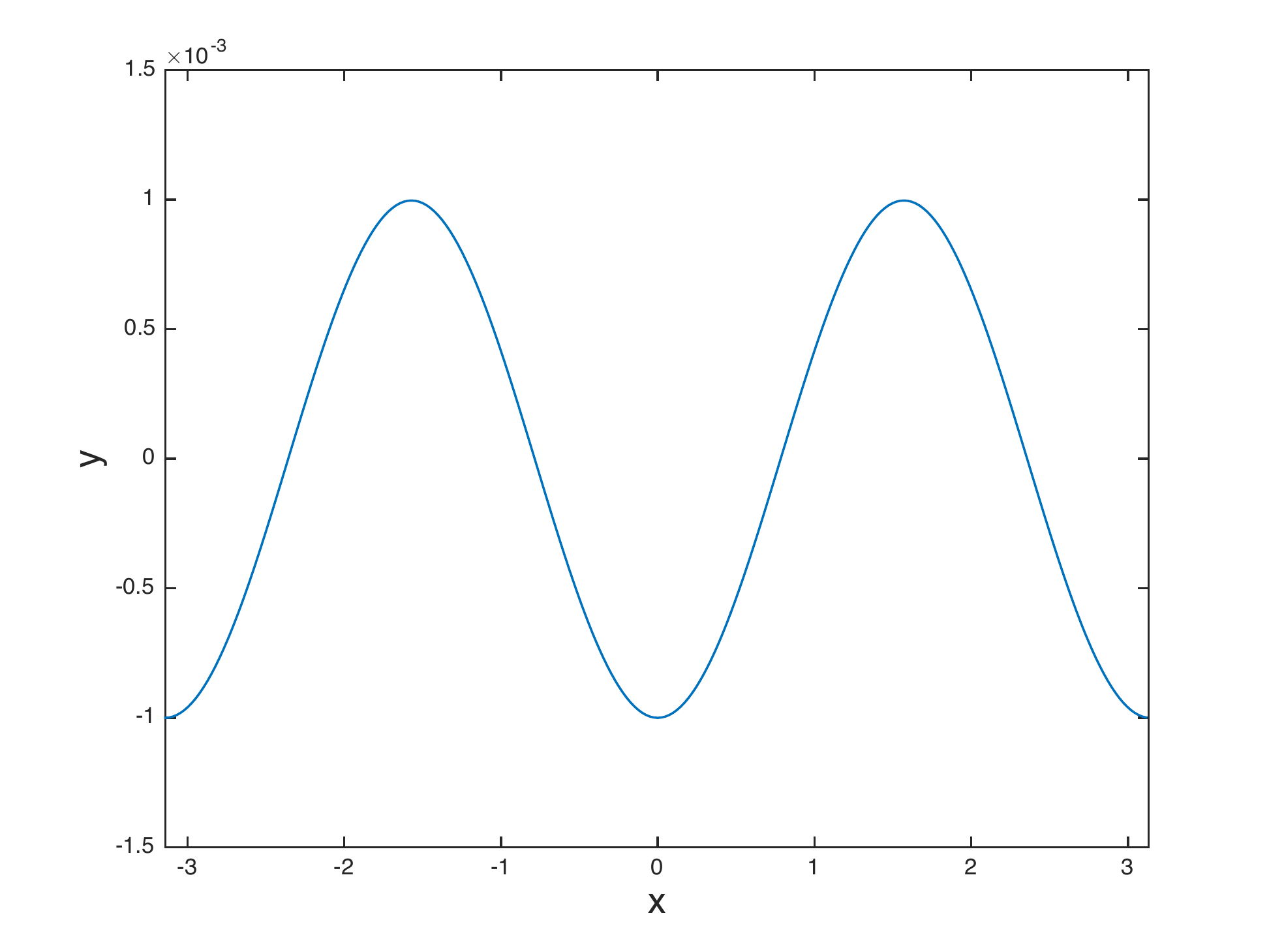}\includegraphics[width=0.5\textwidth]{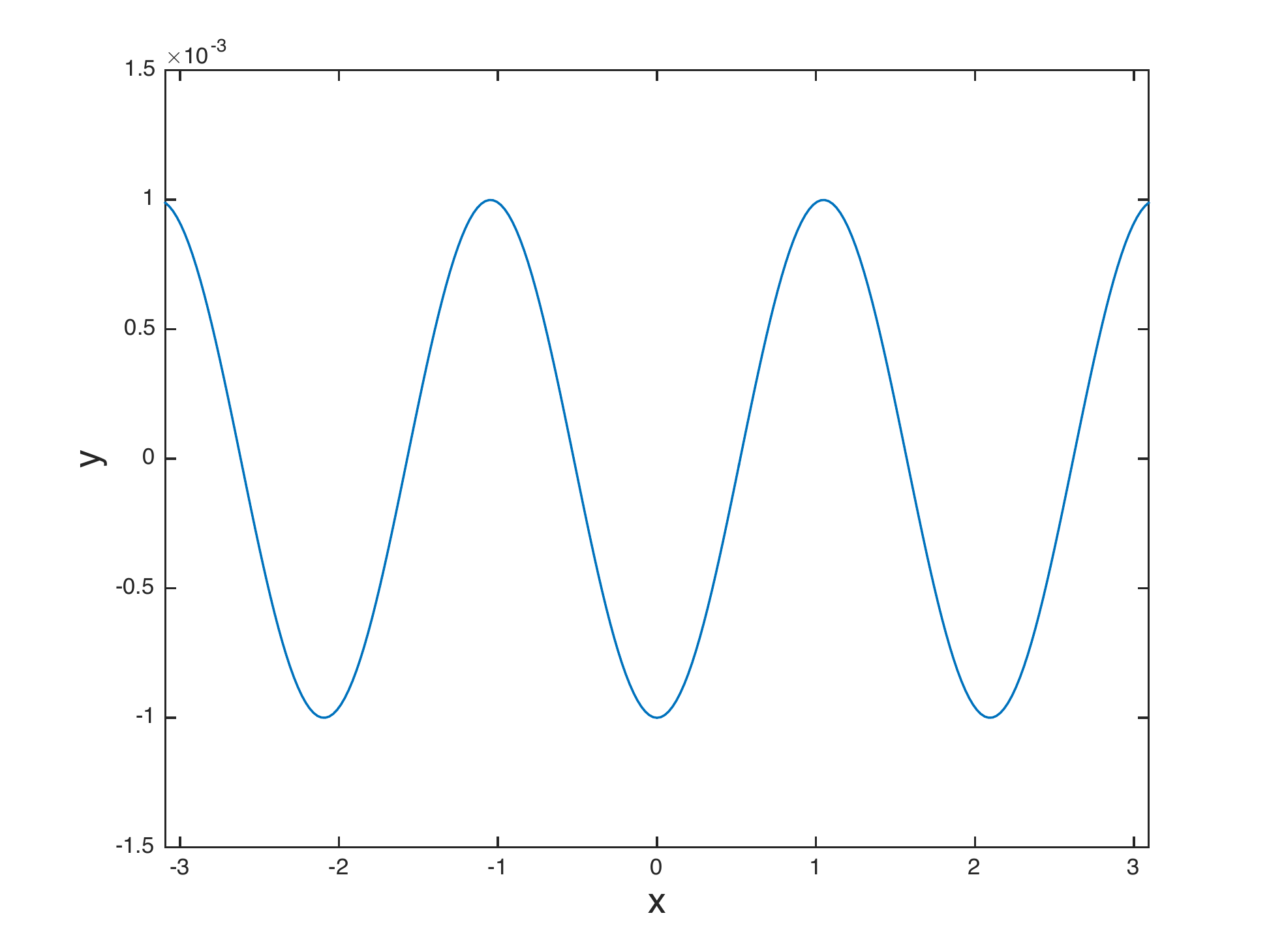}}
\centerline{\includegraphics[width=0.5\textwidth]{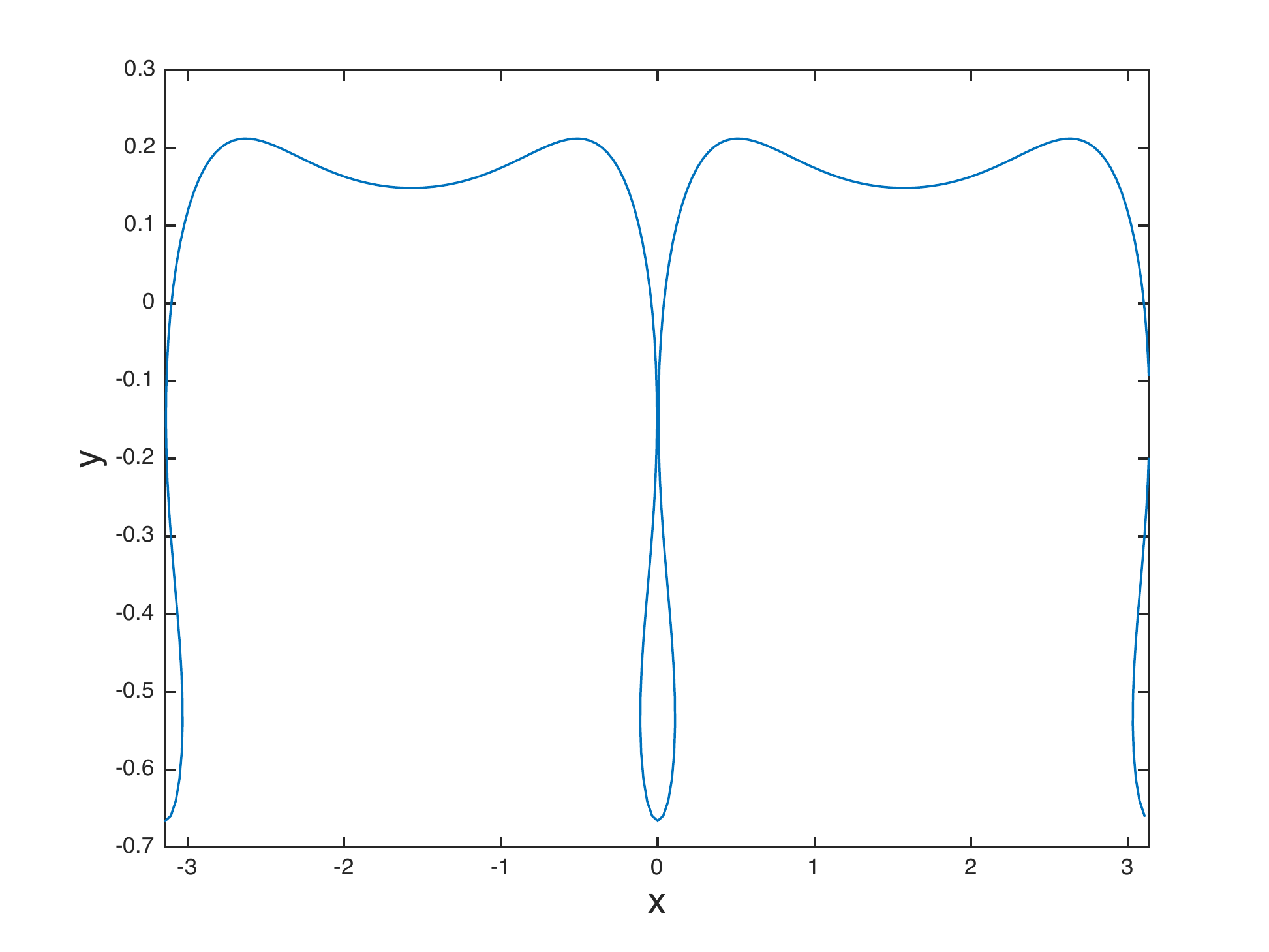}\includegraphics[width=0.5\textwidth]{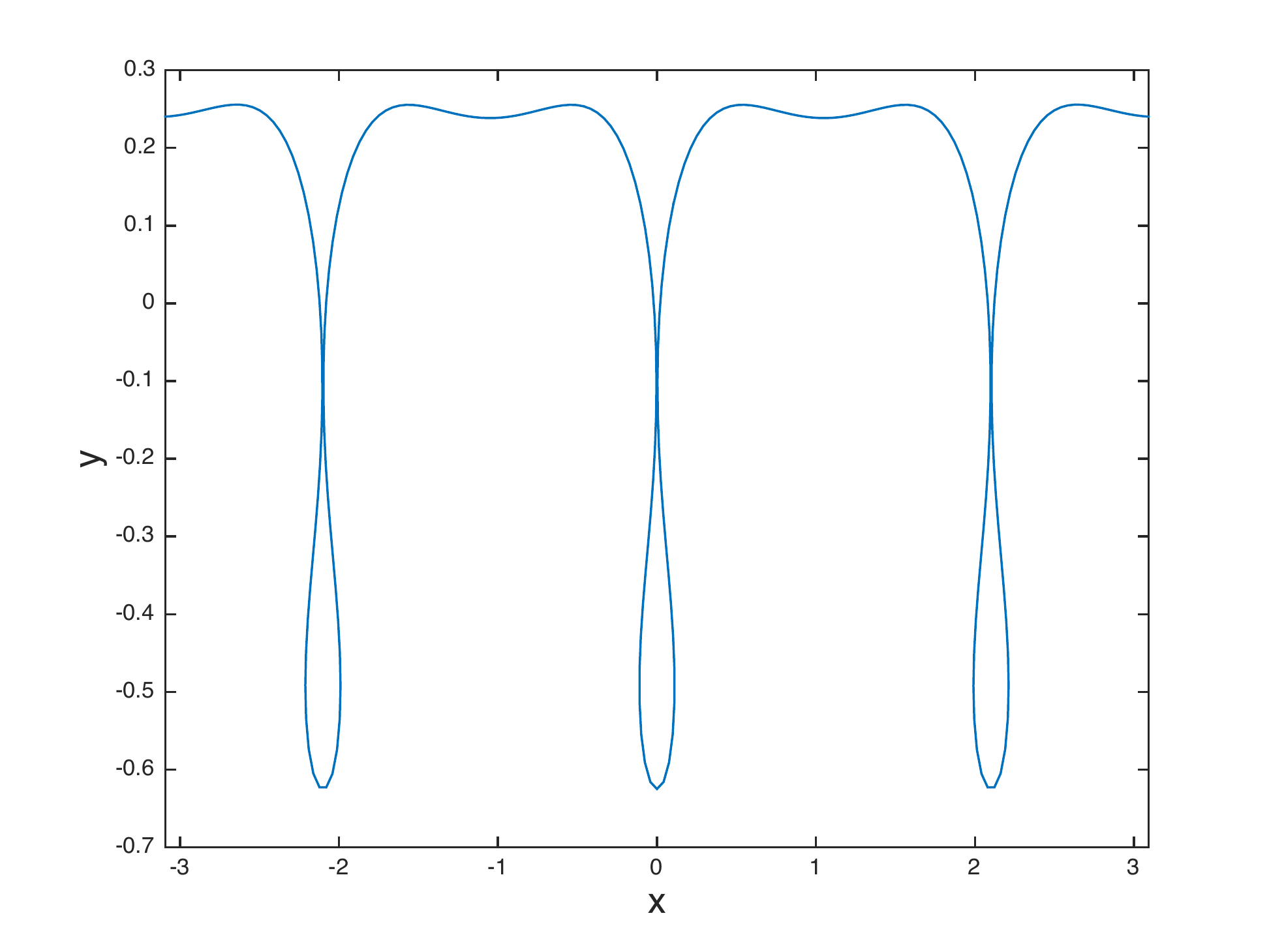}}
\caption{\it Examples of wave profiles from two branches of traveling waves bifurcating from the same speed, at $S=1/63, \  \tilde{A}=1/10, \  A=1, \ \tau=2$. The top row are representative of the wave profiles at small amplitude.  The bottom row are the globally extreme waves from each branch.  \label{ProfilesNonRes}}
\end{figure}

\begin{figure}[tp]
\centerline{\includegraphics[width=0.33\textwidth]{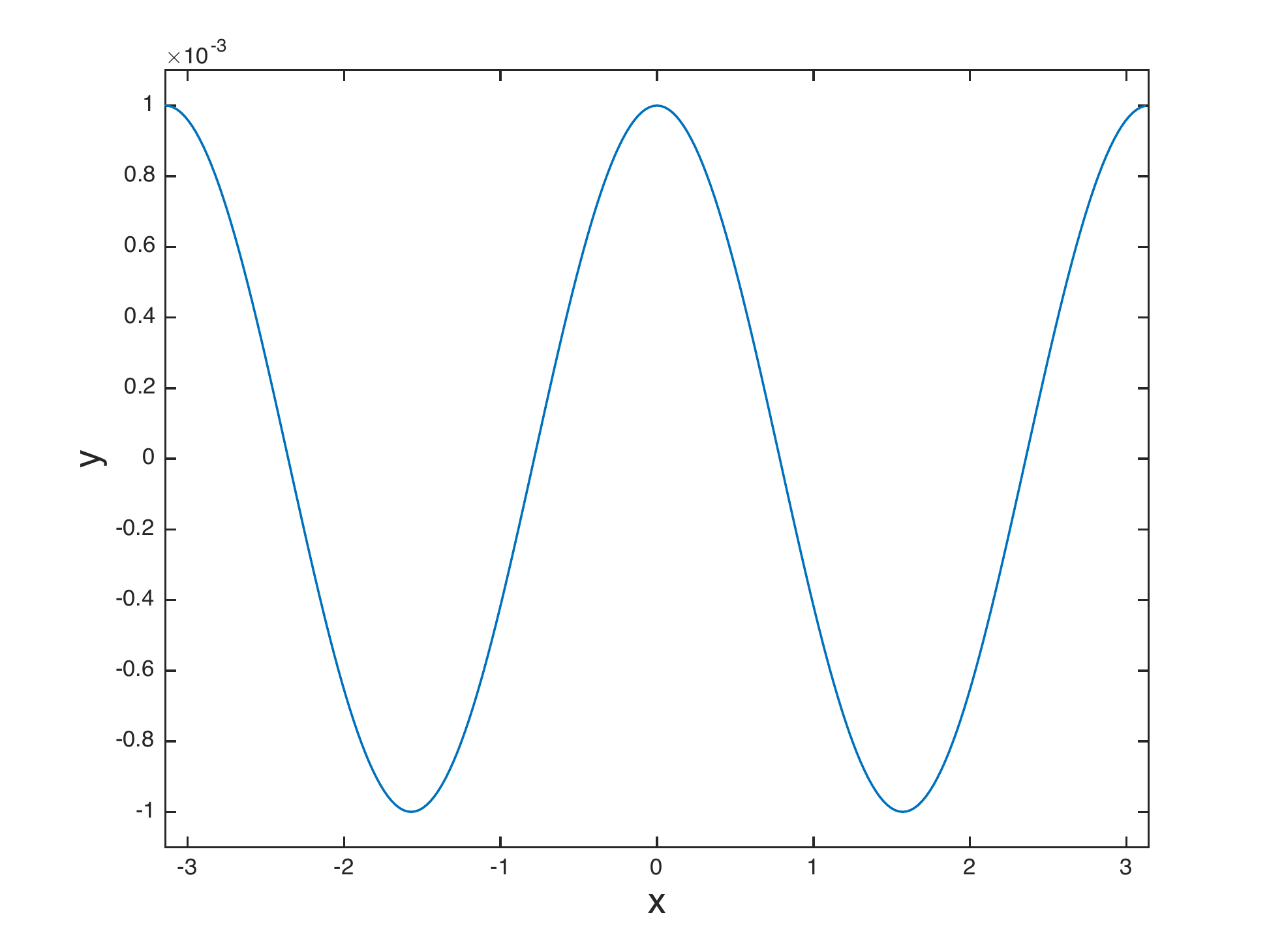}\includegraphics[width=0.33\textwidth]{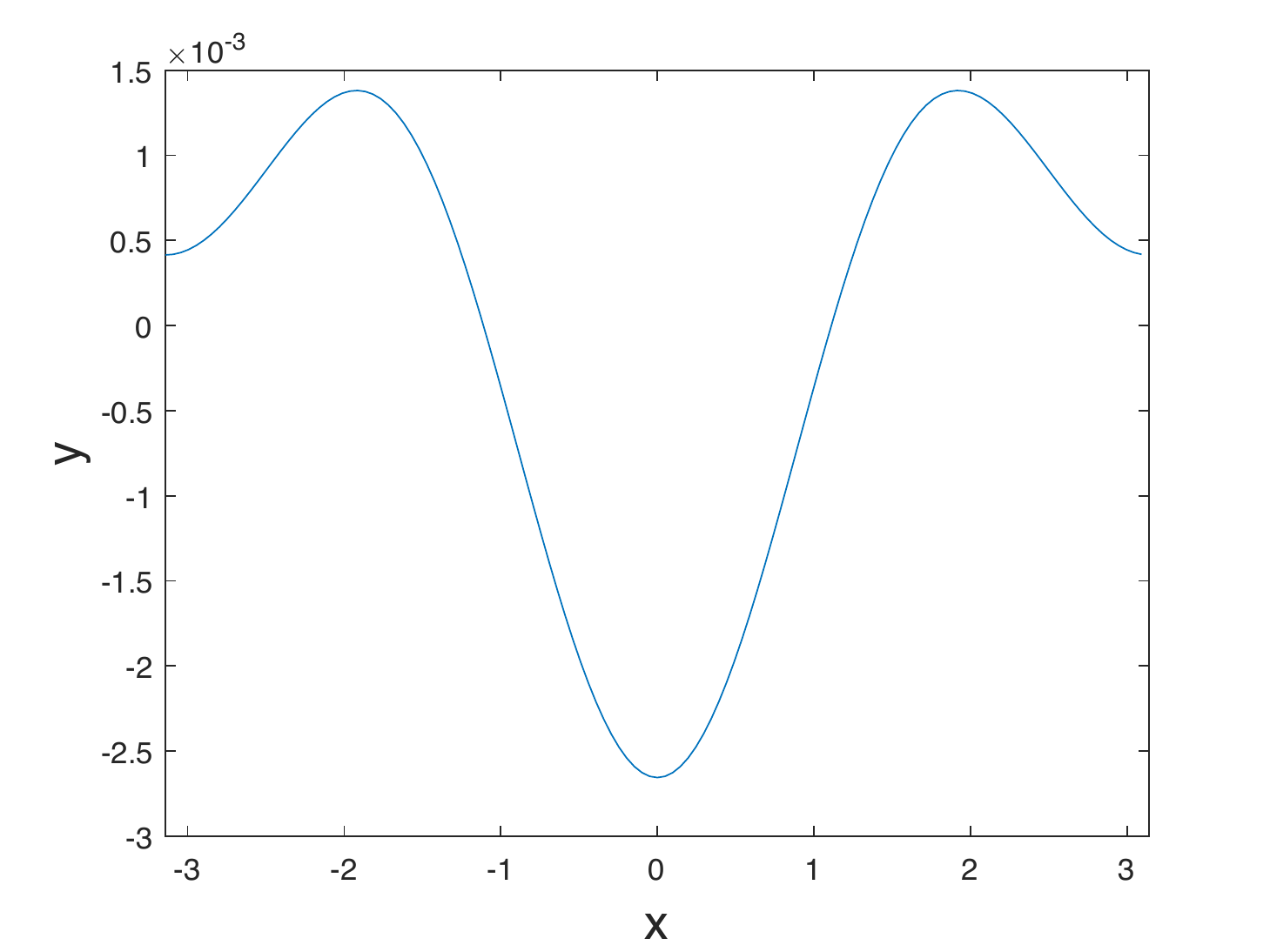}\includegraphics[width=0.33\textwidth]{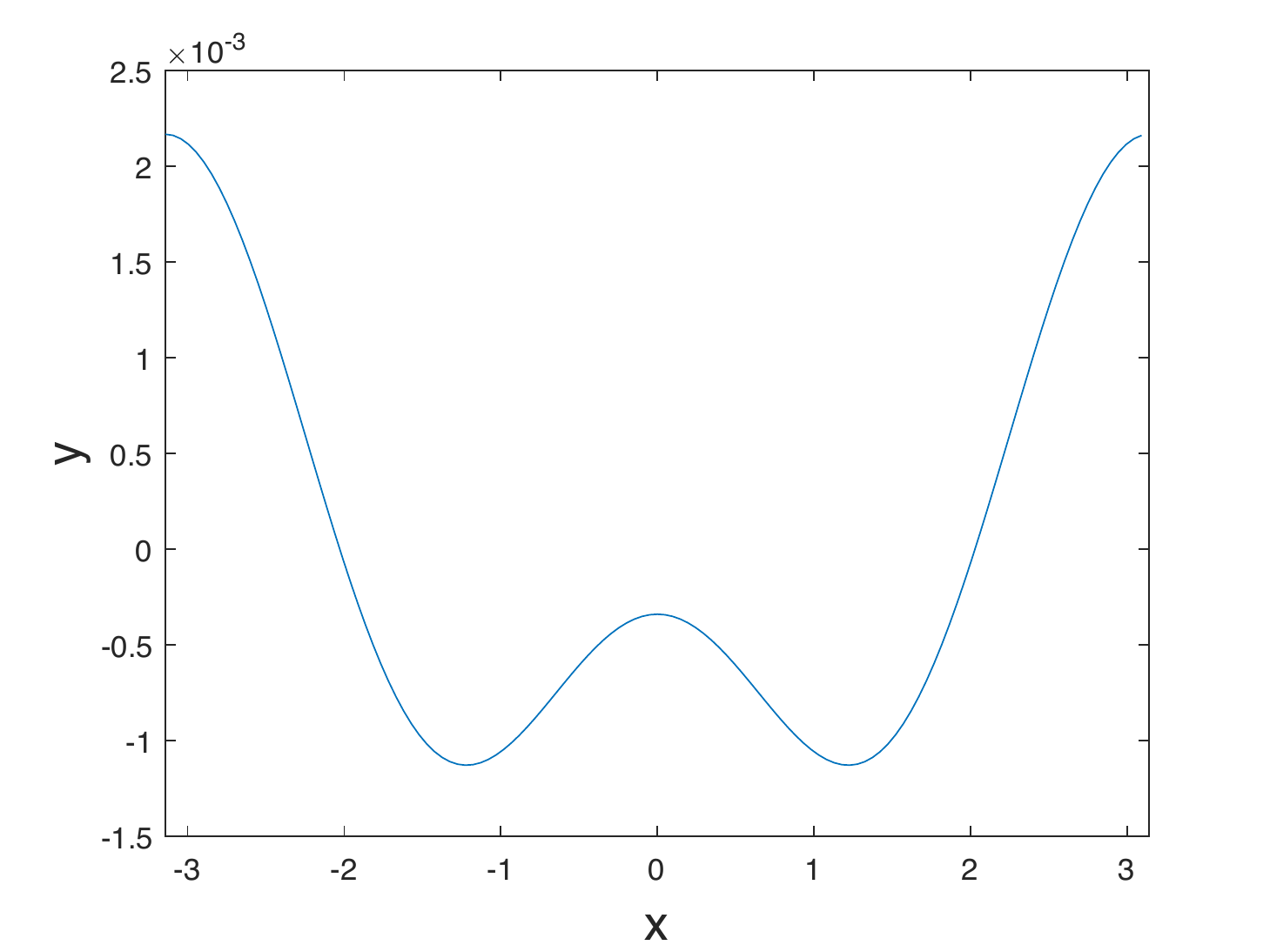}}
\centerline{\includegraphics[width=0.33\textwidth]{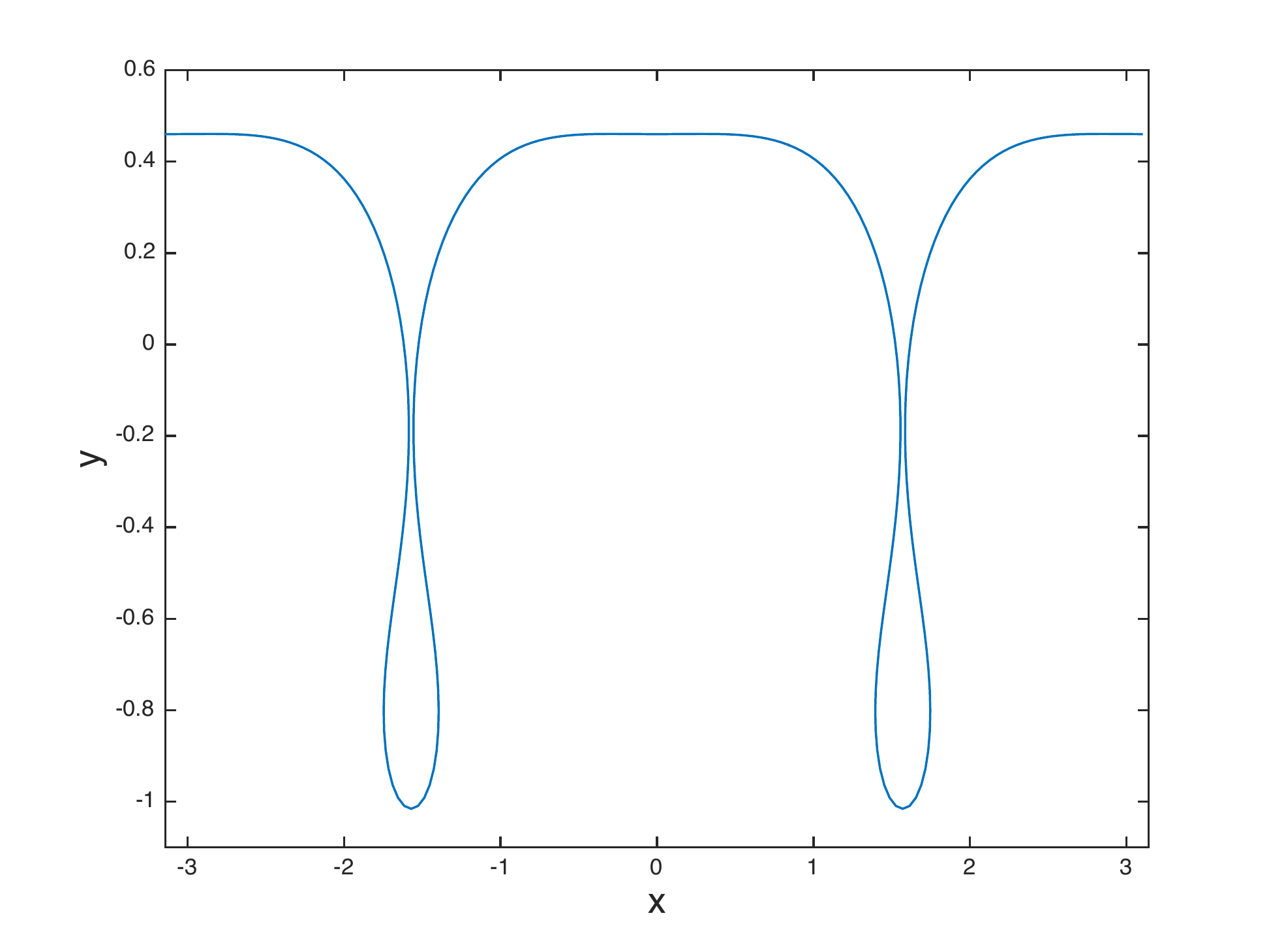}\includegraphics[width=0.33\textwidth]{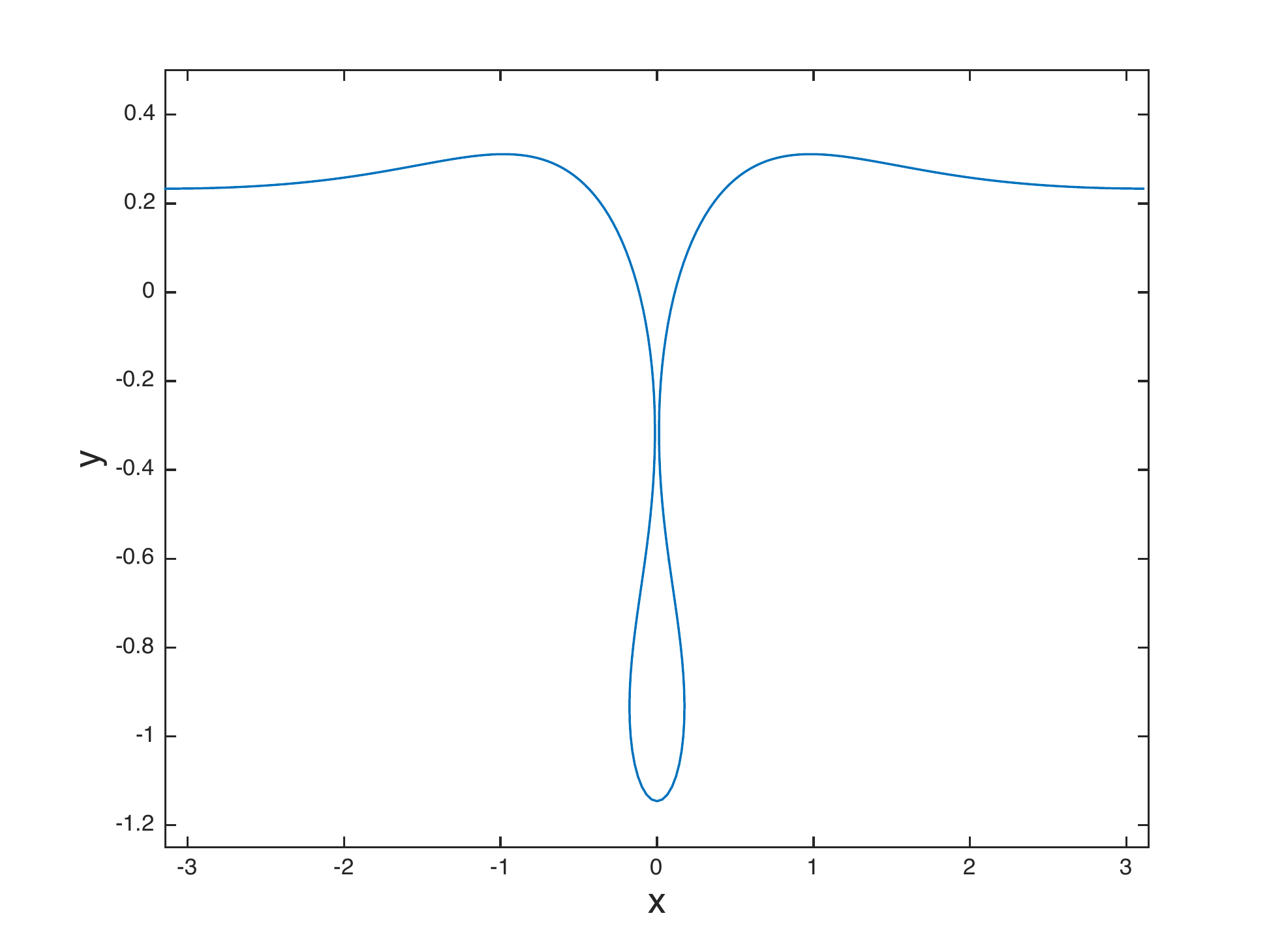}\includegraphics[width=0.33\textwidth]{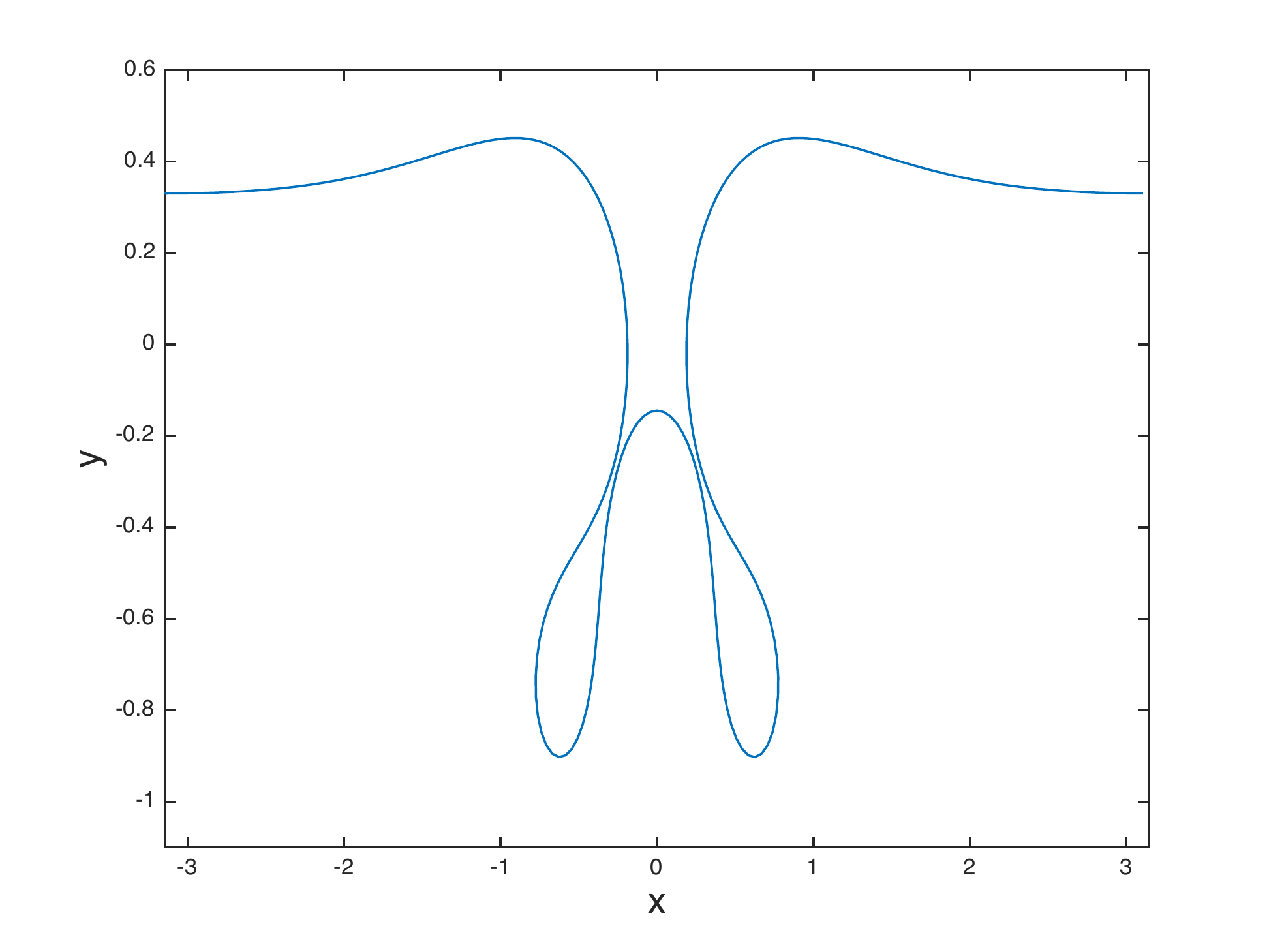}}
\caption{\it Examples of wave profiles from three branches of traveling waves bifurcating from the same speed, at $S=1/9, \  \tilde{A}=1/5, \  A=1, \ \tau=2$. The top row are representative of the wave profiles at small amplitude.  The bottom row are the globally extreme waves from each branch.  The bottom left and center panels are numerical computations of traveling waves which self-intersect; the bottom right panel is a static wave which does not self intersect.  \label{Profiles}}
\end{figure}

Most components of (\ref{eq:gammaEqBeforeSigmaTauDivision}), (\ref{eq:thetaEq1}) are trivial to compute in Fourier space, perhaps with the exception of the Birkhoff-Rott integral $W^\ast$. \ To compute $W^\ast$, we use the decomposition (\ref{eq:BirkRottDecomp}), i.e.
\begin{equation*}
W^{\ast }=\frac{1}{2 i}H\left( \frac{\gamma }{z_{\alpha }}\right) +\mathcal{K} \left[ z%
\right] \gamma \text{,}
\end{equation*}%
where $H$ again denotes the Hilbert transform and the remainder $\mathcal{K}\left[ z%
\right] \gamma $ is the integral%
\begin{eqnarray*}
&&\mathcal{K}\left[ z\right] \gamma \left( \alpha \right) =  \\
&&\frac{1}{4\pi i}\func{PV}\int\nolimits_{0}^{2\pi }\gamma \left( \alpha
^{\prime }\right) \left[ \cot \left( \frac{1}{2}\left( z\left( \alpha
\right) -z\left( a^{\prime }\right) \right) \right) -\frac{1}{\partial
_{\alpha ^{\prime }}z\left( \alpha ^{\prime }\right) }\cot \left( \frac{1}{2}%
\left( \alpha -\alpha ^{\prime }\right) \right) \right] d\alpha ^{\prime }.
\end{eqnarray*}%
The Hilbert transform is trivial to compute; its Fourier symbol is simply $-i \func{sgn} \left( k\right)$. \ To compute the integral $\mathcal{K}\left[ z%
\right] \gamma $, we use an ``alternating" version of the trapezoid rule (i.e. to evaluate the integral at an ``odd" grid point, we sum over ``even" nodes, and vice-versa) \cite{AkersAmbroseSulonPeriodicHydroelasticWaves}.

The $2N + 1$ algebraic equations are then numerically solved with the quasi-Newton method due to Broyden \cite{Broyden}.  For all branches an initial guess is required.  We choose initial guesses based on the nature of the small amplitude asymptotics of the traveling waves.   For Stokes' waves, a sinusoid with the correct phase speed is sufficient.  For Wilton ripples, an initial guess is made using the asymptotics discussed in Section \ref{sec:wiltonRipples}, using the sign of the speed correction, $c_1$ from equation \eqref{WiltonC1}, to choose branches.    

Wilton ripples, where the infinitesimal waves on the branch are supported at two wave numbers, exist only when both the linearization has a 
two-dimensional kernel and the wave numbers in this kernel satisfy $l \mathbin{/} k \in \mathbb{N}$.   Stokes waves can exist regardless of either of these conditions, and in fact we compute Stokes waves in both the resonant and non-resonant situation.   In other words, we observe only Stokes waves in the non-resonant case (when the linearization has one- or two-dimensional kernel) and both Stokes waves and Wilton ripples in the resonant case, where the Stokes wave is supported at the higher wavenumber $l$.  Stokes wave computations in the non-resonant case are computed as illustrations of the 
existence theorem.  Stokes waves can also bifurcate from the same linear speed as a pair of Wilton ripples, however this situation is not governed by our theorem.  

We chose two example configurations for which 
to present computations.  All computations took $\tau=2, A=1$. We then chose the value of $S$ to get the 
two-dimensional kernel for some speed.  When $S=1/9$, the wave numbers $(k,l)=(1,2)$ travel at the same speed, a resonant case. The bifurcation structure here includes two Wilton ripples supported at both wave numbers at leading order, and one Stokes wave supported only at even wave numbers.  Generally if $l$ is the larger of the two frequencies involved in the resonant wave, there is both a pair of Wilton ripples and a Stokes wave supported at frequencies $nl$, \ $n\in \mathbb{N}$. We also simulate at $S=1/63$, where two wavenumbers $(k,l)=(2,3)$ travel at the same speed;  the bifurcation structure here is made up of two Stokes waves.  This is the ``non-resonant" situation of the theorem.

Branches of traveling waves are computed via continuation.  For small amplitude waves, total displacement $h=\max(y)-\min(y)$, is used as a continuation parameter.  As amplitude increases turning points in displacement are circumvented by switching continuation parameters, to amplitude of a Fourier mode of the solution.  The procedure is automated, increasing the wave number used for continuation each time evidence of a turning point in that parameter, or step size below a tolerance, is observed.  Since the amplitude of a fixed Fourier harmonic need not be increasing along a branch, we track the direction the amplitude of the harmonic is changing to prevent retracing previous computations.  For the Wilton ripples, we use consecutive harmonics. The Stokes waves computed here are not supported at all wavenumbers; wavenumbers with zero support are skipped.  

We denote a wave as \emph{globally extreme} when the continuation method has reached, to numerical precision, the termination criterion of self-intersection.  Note that such globally extreme waves represent the end of the branch for fixed parameter values, and not necessarily the furthest solution from the flat state (in some norm).  Alternatively, some of our computed branches terminate in a static wave (i.e. a solution with $c=0$), which can also be considered a return to trivial, since static waves occur at locations where branches with positive and negative speeds collide.

\begin{figure}[tp]
\centerline{\includegraphics[width=0.5\textwidth]{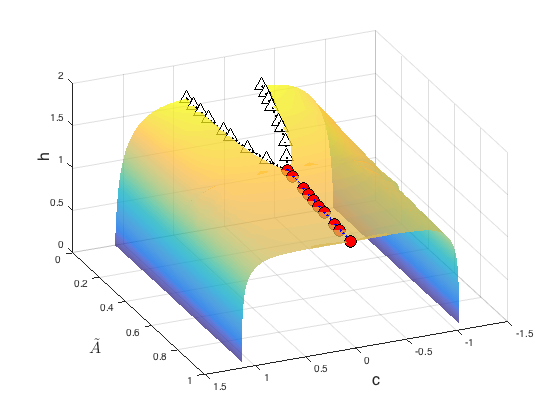}\includegraphics[width=0.5\textwidth]{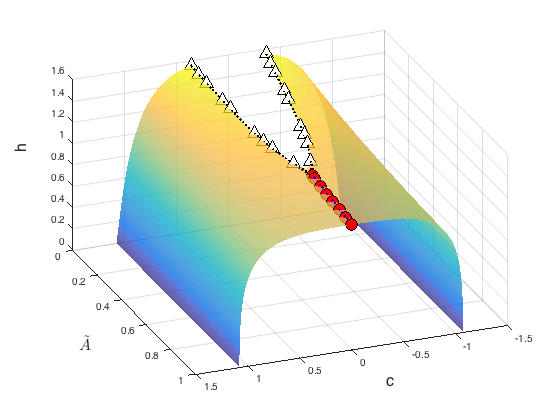}}
\caption{\it Bifurcation surfaces, collections of global branches of traveling waves, are depicted in speed $c$, mass $\tilde{A}$, and total displacement $h=\max(y)-\min(y)$.  Both surfaces have $\tau=2,\  A=1, \text{and} \ S=1/9$.  The surface in the left panel is composed of Stokes' waves, the left panel of Figure \ref{Profiles}.  The right panel is composed of Wilton ripples, the branch in the center of Figure \ref{Profiles}.  The filled red circles mark when branches of traveling waves terminate in static waves; the empty white triangles mark when branches terminate in self intersection.  \label{Surfaces}}
\end{figure}

The numerical computations presented consider Wilton ripple resonances, where the null space of the linear operator has dimension two, including both a wave number and its harmonic, $k=1$ and $k=n$.  For water waves this happens at a countable collection of Bond numbers.  For the hydroelastic case there is a four parameter family, satisfying
\begin{equation}\label{RippleS} S=\frac{2A}{n(n^2+n+1+\tau)} \end{equation}
In \eqref{RippleS}, $A$, $\tau$ and $S$ are real valued, and positive, $n\in \mathbb{N}$.  Thus for any pair $A$ and $\tau$, there is a countable number of such resonances.  Notice in \eqref{RippleS} that the parameter $\tilde{A}$ does not appear, thus these resonances occur independent of the mass $\tilde{A}$.  As discussed in earlier sections, we focus here on the second harmonic resonance, where the wave numbers of the infinitesimal solution are $k=1$ and $k=2$ (other configurations exist, see \cite{akers2012wilton,trichtchenko2016instability,vanden2002wilton,mcgoldrick1970wilton,akers2016high,OlgaPreprint}).  For each class, the numerical method begins with an initial guess. 

\begin{figure}[tp]
\centerline{\includegraphics[width=0.5\textwidth]{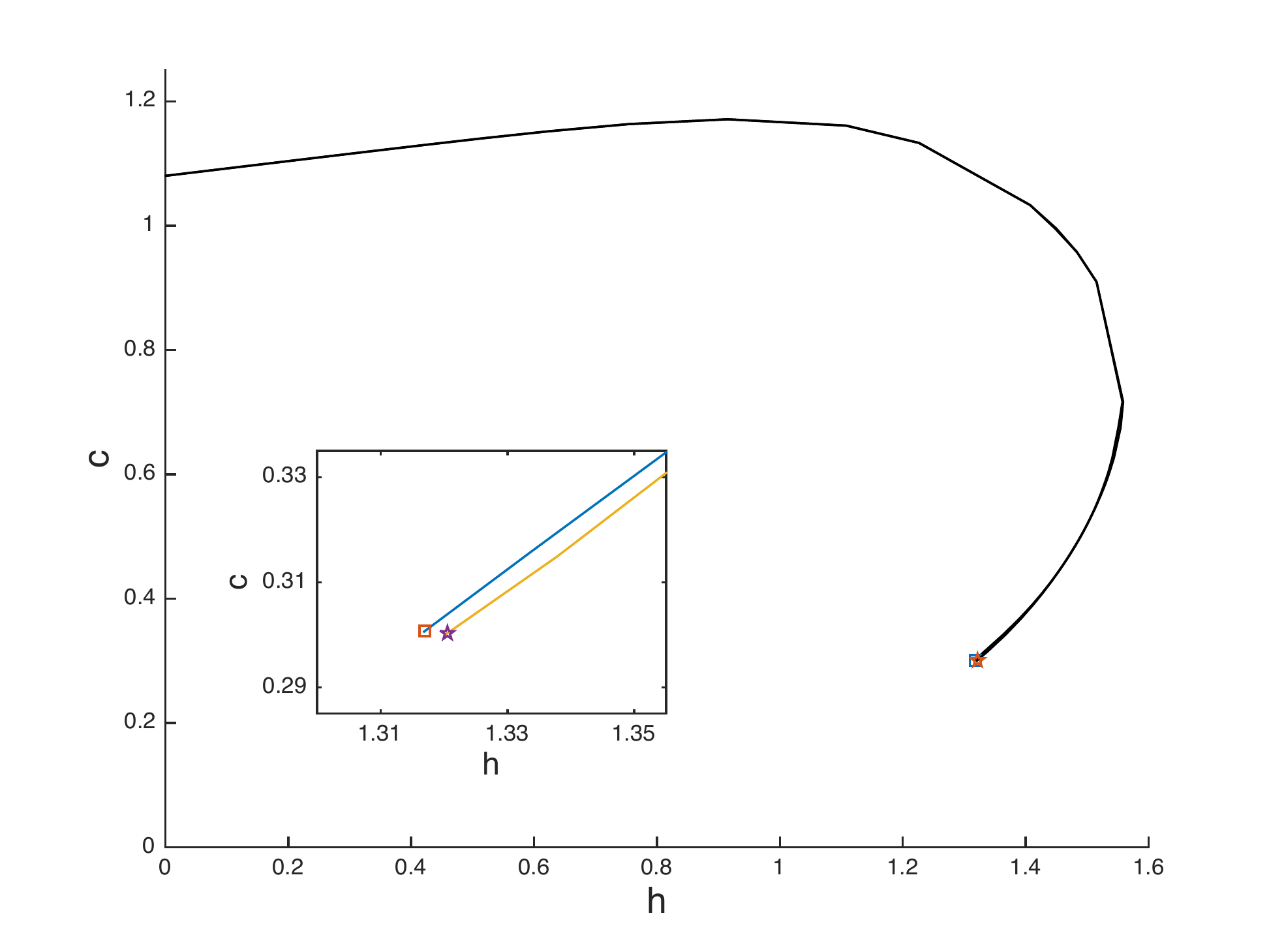}\includegraphics[width=0.5\textwidth]{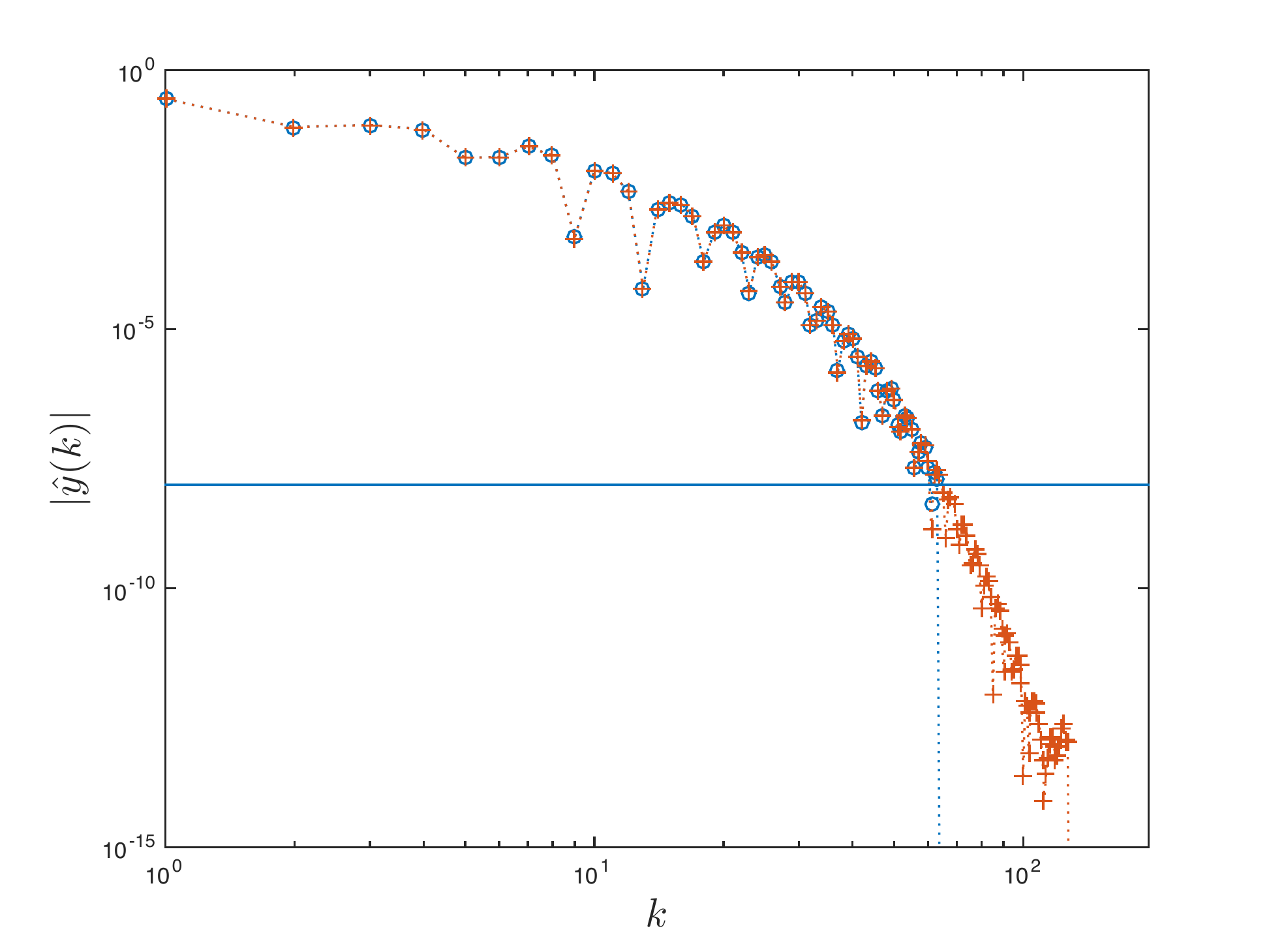}}
\caption{\it The convergence data for a sample configuration, at $S=1/9, \tilde{A}=1/100, A=1, \tau=2$ is depicted.  On the left, computed speed-amplitude curves at two different spatial resolutions are compared ($N=128$ to $N=256$).  At the level reported in Figure \ref{Branches}, these curves agree; closer inspection reveals that by the extreme configuration, inset, the curves differ on the order of $10^{-3}$.    On the right, the Fourier modes of the extreme wave are depicted.  With $N=128$, marked with circles, the highest frequency Fourier modes have decayed almost to the error tolerance of the quasi-Newton iteration, marked with a solid line.  When $N=256$, marked with plus signs,  the included Fourier modes decay to near machine precision.  \label{Convergence}}
\end{figure}

In this work, we present example simulations for $A=1,\ \tau=2$, $n=2$ and $S=1/9$.  In this configuration, three branches of traveling waves bifurcate from the flat state, two Wilton ripples, whose asymptotics are in the previous section, and one Stokes wave, whose small  amplitude solutions are supported only at $k=2$.  Example speed-amplitude curves and wave profiles on each branch are presented in Figures \ref{Branches} and \ref{Profiles} respectively.   We observe that both Wilton ripples and Stokes waves have configurations which terminate in self intersection as well as in static waves.

In addition to computing individual branches of traveling waves, we consider the dependence of these branches on the mass parameter $\tilde{A}$.  Linear solutions do not depend on this parameter; this dependence on $\tilde{A}$ is entirely a nonlinear effect.  In Figure \ref{Surfaces}, we present bifurcation surfaces upon which traveling waves exist for both the Stokes wave case, and the depression Wilton ripple (where the center point is a local minimum).  In each case we see that, for small $\tilde{A},$ branches terminate in self-intersection, until a critical value after which the 
extreme wave on a branch is a static wave.   

The majority of the computations presented here used $N=128$ or $N=256$ points equally spaced in arc-length to discretize the interface, and a tolerance of $10^{-8}$ for the residual of the quasi-Newton solver.  The Fourier spectrum an extreme configuration is reported at these two resolutions in Figure \ref{Convergence}.  Notice that at $N=128$, the spectrum has decayed to approximately the residual tolerance, where at $N=256$ it has decayed far below, near machine precision.  In terms of the profiles and surfaces presented here, these resolutions are indistiguishable.  In the left panel of Figure \ref{Convergence}, the difference in the speed-amplitude curves at these two resolutions is presented.  The upon zooming in at the extreme portion of the speed-amplitude curve we see a departure on the order of $10^{-3}$.  This loss of digits is due to the difficulty of computing a singular integral on a near-intersecting interface.  This difficulty is known, see \cite{helsing2008evaluation}.

\section{Conclusion}
In the present work and \cite{AkersAmbroseSulonPeriodicHydroelasticWaves}, we have studied spatially periodic traveling waves for interfacial hydroelastic waves with and without mass.  Our results include existence theory and computational results in most cases.  The cases stem from the dimension of the kernel of the relevant linearized operator.  We have demonstrated that the linearized operator has an at most two-dimensional kernel, and we treated the one-dimensional case in \cite{AkersAmbroseSulonPeriodicHydroelasticWaves} both analytically and numerically.  In the present work, we treated both non-resonant and resonant cases of a two-dimensional kernel, and also gave a non-rigorous asymptotic treatment and a numerical study of interfacial hydroelastic Wilton ripples as particular configurations in the resonant case.  We aim to provide a proof of existence of such Wilton ripples (maintaining a fixed value of the parameter $\tau_{1}$)  in a future work.

\bibliography{references}{}
\bibliographystyle{plain}

\end{document}